\documentclass[12pt, reqno]{amsart}

\usepackage{amsfonts}
\usepackage{amssymb}
\usepackage{amscd}
\usepackage{amsmath}

\usepackage{tikz}
\usetikzlibrary{calc}

\usepackage{subcaption}

\usetikzlibrary{decorations.markings, arrows.meta, arrows}
\tikzset{
  equals/.style={double=none, double distance=2pt}, 
  cdarr/.style={->},
  quivarr/.style={-latex,thick},
}   
\usepackage{color}

\input xy
\xyoption{all}

\numberwithin{equation}{section}

\title{Grassmannian cluster subcategories and positroid varieties}
\author{Bernt Tore Jensen}
\address{Department of Mathematical Sciences, Norwegian University of Science and Technology, 
Gj\o vik, Teknologvn. 22, 2815 Gj\o vik, Norway}
\email{bernt.jensen@ntnu.no}

\author{Liam Riordan}
\address{Mathematical Sciences, University of Bath, Bath BA2 7AY, U.K.}
\email{lrr27@bath.ac.uk}
\author{Xiuping Su}
\address{Mathematical Sciences, University of Bath, Bath BA2 7AY, U.K.}
\email{xs214@bath.ac.uk}

\usepackage{amssymb,amsthm,array}

\usepackage{lmodern}
\usepackage{verbatim}

\usepackage{amsthm}

\usepackage[linktocpage,pdfstartview=FitH,colorlinks=true,
linkcolor=blue,citecolor=magenta]{hyperref}

\theoremstyle{definition}
\newtheorem{theorem}{Theorem}[section]
\newtheorem{proposition}[theorem]{Proposition}
\newtheorem{lemma}[theorem]{Lemma}
\newtheorem{corollary}[theorem]{Corollary}
\newtheorem{remark}[theorem]{Remark}
\newtheorem{example}[theorem]{Example}
\newtheorem{conjecture}[theorem]{Conjecture}

\theoremstyle{definition}
\newtheorem{definition}[theorem]{Definition}

\newcommand{\minor}[1]{\Delta_{#1}}

\newcommand{\ZZ}{\mathbb{Z}}
\newcommand{\QQ}{\mathbb{Q}}
\newcommand{\CC}{\mathbb{C}}
\newcommand{\NN}{\mathbb{N}}

\newcommand{\ol}[1]{\overline{#1}}

\newcommand{\Hom}{\mathrm{Hom}}

\newcommand{\End}{\operatorname{End}}
\newcommand{\sub}{\operatorname{Sub}}
\renewcommand{\mod}{\operatorname{mod}}
\newcommand{\fac}{\operatorname{Fac}}
\newcommand{\Gr}{\operatorname{Gr}}
\newcommand{\GP}{\operatorname{GP}}

\newcommand{\proj}{\operatorname{proj}}

\newcommand{\rep}{\operatorname{Rep}}
\newcommand{\add}{\operatorname{add}}

\newcommand{\supp}{\operatorname{supp}}
\newcommand{\Ext}{\operatorname{Ext}}

\newcommand{\cok}{\operatorname{Cok}}
\renewcommand{\ker}{\operatorname{Ker}}
\newcommand{\im}{\operatorname{Im}}

\newcommand{\soc}{\operatorname{soc}}

\newcommand{\lra}{\longrightarrow}

\newcommand{\ra}{\rightarrow}
\newcommand{\sdp}{\times\kern-.2em\vrule height1.1ex depth-.05ex}
\newcommand{\epi}{\lra \kern-.8em\ra}
\newcommand{\ul}{\underline}

\newcommand{\g}{\widehat{g}}
\newcommand{\h}{\mathfrak{h}}

\newcommand{\tM}{\widetilde{M}}
\newcommand{\M}{\widehat{M}}

\renewcommand{\h}{\widehat{h}}

\newcommand{\maxWS}{\mathcal{S}}
\newcommand{\maxNC}{\maxWS}
\newcommand{\Irr}{\operatorname{Irr}}

\newcommand{\ptf}{\mathfrak{P}}

\newcommand{\verteq}{\rotatebox{90}{$\,=$}}

\newcommand{\X}{\widehat{X}}

\newcommand{\rk}{\operatorname{rk}}

\newcommand{\CM}{\operatorname{CM}}

\newcommand{\Bbar}{\pi B}

\renewcommand{\int}{\mathrm{int}}

\newcommand{\rank}{\operatorname{rank}}

\newcommand{\catC}{\mathcal{C}}
\newcommand{\catD}{\mathcal{D}}
\newcommand{\catP}{\mathcal{P}}

\newcommand{\catB}{\mathcal{B}}
\newcommand{\cmax}{\mathcal{S}'}

\newcommand{\opoid}{ \overset{\circ}{\Pi}}
\newcommand{\poid}{\Pi}

\newcommand{\clalg}{\mathsf{A_{clu}}}
\newcommand{\clalgptf}{\mathsf{A^{\ptf}_{clu}}}
\newcommand{\clalgptfloc}{(\mathsf{A^{\ptf}_{clu}})_{B^{-1}}}

\newcommand{\algGP}{\mathsf{A}_{\mathrm{GPB}}}
\newcommand{\algCM}{\mathsf{A_{CMB}}}
\newcommand{\algGPptf}{\mathsf{A^{\ptf}_{GPB}}}
\newcommand{\algCMptf}{\mathsf{A^{\ptf}_{CM}}}
\newcommand{\algD}{\mathsf{A}_{\catD}}
\newcommand{\algDd}{\mathsf{A}_{\catD_2}}
\newcommand{\ideal}{\mathsf{I}}
\newcommand{\algMS}{\mathsf{A_{MS}}}

\newcommand{\labset}[1]{[#1]}
\newcommand{\labsubset}[2]{\binom{\labset{#1}}{#2}}

\newcommand{\nlace}{\mathsf{N}}
\newcommand{\positr}{\mathsf{P}}

\newcommand{\openpv}{\stackrel{\circ}{\Pi}}

\newcommand{\flabs}{\mathcal{F}(G)}
\newcommand{\ring}{\mathcal{R}}
\newcommand{\field}{\mathcal{K}}

\newcommand{\spec}{\operatorname{Spec}}
\newcommand{\hC}{\mathbf{c}}
\newcommand{\HC}{\mathbf{C}}
\newcommand{\IrrSub}{\operatorname{Irr}_{\sub}}
\newcommand{\IrrCM}{\operatorname{Irr}_{\CM}}
\newcommand{\preproj}{\Lambda}
\newcommand{\gen}{\mathrm{gen}}
\newcommand{\gminvec}[1]{\overline{\mathrm{g}}_{\gen}(#1)}
\newcommand{\hCmin}{\hC_{\gen}}
\newcommand{\hCminco}{\hC^{\mathrm{co}}_{\mathrm{gen}}}
\newcommand{\GrotAbar}{\Grot(\fd\olA)}
\newcommand{\GL}[1]{\operatorname{GL}(#1)}
\newcommand{\st}{:}
\newcommand{\preprojQ}{Q}
\newcommand{\olA}{\overline{A}}
\newcommand{\olT}{\overline{T}}
\newcommand{\olM}{\overline{M}}
\newcommand{\op}{^{\mathrm{op}}}
\newcommand{\olQ}{\overline{Q}}
\newcommand{\Grot}{K}
\newcommand{\fd}{\operatorname{fd}}
\newcommand{\dimv}{\operatorname{\underline{dim}}}
\newcommand{\gminmap}{\overline{\mathrm{g}}_{\gen}}
\newcommand{\cogv}{\operatorname{cog}}
\newcommand{\gv}{\operatorname{g}}

\newcommand{\bas}{\mathfrak{B}}
\newcommand{\basGP}{\mathfrak{B}_{\mathsf{GPB}}}
\newcommand{\basCM}{\mathfrak{B}_{\mathsf{CMB}}}
\newcommand{\basD}{\mathfrak{B}_{\catD}}

\newcommand{\gi}{\operatorname{GI}}
\newcommand{\rich}{\operatorname{R}(v,w)}
\newcommand{\bl}[1]{\textcolor{black}{#1}}

\renewcommand{\op}{\operatorname{^{\mathrm{op}}}}

\newcommand{\ggop}{\ol{\mathrm{g}}^{\op}_{\gen}}
\newcommand{\gcog}{\ol{\mathrm{cog}}_{\gen}}
\newcommand{\wt}{\operatorname{wt}}
\newcommand{\sn}{\mathsf{S_n}}

\newcommand{\set}{\mathsf{S}}

\newcommand{\positrij}{\mathsf{P}_{1}}
\newcommand{\positrji}{\mathsf{P}_{2}}
\newcommand{\maxNCij}{\mathcal{S}_{1}}
\newcommand{\maxNCji}{\mathcal{S}_{2}}
\newcommand{\nlaceij}{\mathsf{N}_{1}}
\newcommand{\nlaceji}{\mathsf{N}_{2}}
\newcommand{\maxNCone}{\widehat{\mathcal{S}}_{1}}
\newcommand{\maxNCtwo}{\widehat{\mathcal{S}}_{2}}

\normalbaselines
\begin{document}
\maketitle

\begin{abstract}
A class of subcategories $\GP B$ of the Grassmannian cluster category $\CM C_{k, n}$ was
constructed by Jensen--King--Su from certain superorders $B$ of $C_{k, n}$, which they showed
are  in bijection with Grassmannian positroids of type $(k, n)$.
We prove that $\GP B$ admits a cluster substructure of $\CM C_{k, n}$, giving rise to
a cluster algebra $\clalg$.  This naturally raises questions regarding the relationship 
of $\clalg$ to $\CC[\Gr(k, n)]$ and 
to the coordinate ring of the positroid variety associated to $B$.

Using the cluster substructure, we show
that the ice Gabriel quiver $Q^\circ_U$ of a cluster tilting object $U\in \GP B$, consisting of rank one modules,  
is a subquiver of $Q^\circ_T$ with $T$ a cluster tilting object in $\CM C_{k, n}$ containing $U$ as a summand. We also 
deduce that $\clalg$ is a subalgebra of $\CC[\Gr(k, n)]$. 
Moreover, applying a result of Canakci--King--Pressland on the Gabriel quiver
$Q_U$ in the case where $B$ is connected (i.e., 
has no repeated direct summands), 
we deduce that $Q^\circ_U$, for arbitrary $B$, coincides with the quiver constructed by Muller-Speyer from 
a plabic graph whose face labels agree with the indices of the indecomposable summands of $U$.
Consequently, upon localisation at the minors corresponding to $B$, the algebra
$(\clalg)_B$ is isomorphic to the cluster algebra $\algMS$ of Muller-Speyer.
We then construct bases for certain 
subalgebras and for an ideal of $\CC[\Gr(k, n)]$, whose elements have pairwise distinct 
minimal and distinct maximal leading exponents, and apply these to prove that 
$(\clalg)_B$ is naturally isomorphic to the coordinate ring of the open positroid variety.

As a consequence, we obtain a new proof of Galashin--Lam's Theorem, identifying $\algMS$ with
the coordinate ring of the open positroid variety, 
which was originally conjectured by Muller-Speyer.
In the connected case, we note also that Pressland gave a categorification of the cluster 
structure following Galashin-Lam. 
\end{abstract}

\setcounter{tocdepth}{1}
\tableofcontents

\section{Introduction}
\subsection{The Grassmannian cluster category $\CM C$}\label{sec:catgr}

Let $\ring$ be the complete local ring $\ring= \CC[[t]]$ and let $\field=\ring[t^{-1}]$
be its field of fractions. Denote by $\Gr(k, n)$  (the affine cone of) the Grassmannian 
of $k$-dimensional quotient spaces of $\CC^n$. In \cite[\S3]{JKS1}, Jensen--King--Su 
introduced an algebra $C=C_{k, n}$ and applied the category $\CM C$ of Cohen-Macaulay 
$C$-modules to give a categorification of  the Grassmannian cluster algebra $\CC[\Gr(k, n)]$ 
constructed by Scott \cite{Scott}. The algebra $C$ can be described as the path 
$\ring$-algebra of the circular double quiver with $n$ vertices labeled by $\ZZ_n$, 
with relations $xy=t=yx$ and $x^k=y^{n-k}$ at each vertex, where $x$ and~$y$ label 
clockwise and anti-clockwise arrows, respectively. For instance,  when $(k,n)=(2,5)$,  

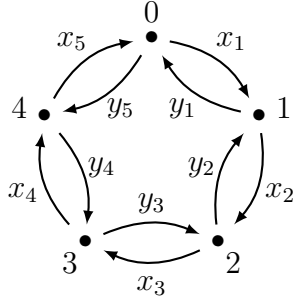
\[\begin{tikzpicture}[scale=1.0,baseline=(bb.base)]
\draw (0,0) node (bb) {};
\pgfmathsetmacro{\stang}{90}
\pgfmathsetmacro{\incang}{72}
\newcommand{\vrad}{1.5}
\newcommand{\nrad}{1.85}
\newcommand{\xrad}{1.8}
\newcommand{\yrad}{0.7}
\foreach \j in {1,...,5}
{ \coordinate (v\j) at (\stang+\incang-\incang*\j:\vrad);
  \draw (v\j) node (v\j) {\small $\bullet$};  }  
\foreach \t/\h in {5/1,1/2, 2/3, 3/4, 4/5}
{ \draw[quivarr] (v\t) to [bend left=25] (v\h);
  \draw[quivarr] (v\h) to [bend left=23] (v\t);
  \draw (\stang+0.5*\incang-\incang*\h:\xrad) node {$x_{\h}$};
  \draw (\stang+0.5*\incang-\incang*\h:\yrad) node {$y_{\h}$};  }
\foreach \j in {0,...,4}
  \draw (\stang-\incang*\j:\nrad) node {$\j$};
\end{tikzpicture}
\]
the relations starting at vertex 0 are  $x_5y_5=t=y_1x_1$ and $x_2x_1=y_3y_4y_5$.
We recall some key properties of the Grassmannian cluster category 
$\CM C$ from \cite{JKS1}. 

There is an exact functor  
\[
   \pi\colon \CM C\to \sub Q_{k},
\]
which induces an equivalence 
\[ 
   \CM C/\add P_0 \to \sub Q_{k},
\]
where $P_0=Ce_0$ is the projective $C$-module at vertex $0$. Note that $\sub Q_k$ is the 
category of $\Lambda$-modules that can be embedded into modules in $\add Q_k$,  
$Q_k$ is the indecomposable injective $\Lambda$-module with socle at vertex $k$, and 
$\Lambda$ is the preprojective algebra of type $\mathbb{A}_{n-1}$. Also, $\sub Q_k$ 
is a Frobenius (stably) 2-Calabi-Yau (2-CY) category \cite{GLS08}. The functor $\pi$ 
is full exact, that is, for any $M, N\in \CM C$,
\[
   \Ext^1_C(M, N)\cong \Ext^1_{\Lambda}(\pi M, \pi N).
\]
Consequently, $\CM C$ is also a Frobenius 2-CY category. 

For any $M\in \CM C$, there is a well-defined  rank of  $M\in \CM C$ such that 
\[
   \rank_{C} M=  \rank_{\ring} e_i M
\] 
for any vertex idempotent $e_i\in  C$. Rank one modules are classified by  
$J\in \labsubset{n}{k}$ (see Figure \ref{fig:profile} for an illustration), where $[n]=\{1, \dots, n\}$ and $\labsubset{n}{k}$ is the 
set of $k$-sets contained in $[n]$, i.e., subsets of size $k$. 
Let $\maxWS$ be a 
collection of $k$-sets and let 
\[
   M_{\maxWS}=\oplus_{I \in \maxWS} M_I.
\] 
There is a combinatorial characterisation of the vanishing of extensions between 
rank one modules, that is,  $\Ext^1_C(M_I, M_J)=0$ if and only if $I$ and $J$ 
are non-crossing sets \cite{Scott} (see Definition \ref{def:NC}).  
Moreover, when $\maxWS$ is a maximal collection of (pairwise) non-crossing
$k$-sets,  $M_{\maxWS}$ is a cluster tilting object (equivalently, a maximal rigid 
module) in $\CM C$. 

There is a cluster character (see \cite{Pal}, \cite{FuKeller}, \cite{JKS1}) 
\begin{equation}\label{eq:Psi}
  \Psi\colon \CM C \to \CC[\Gr(k, n)], M\mapsto \Psi_M,
\end{equation}
which homogenises Geiss--Leclerc--Schr\"oer's cluster character defined on $\sub Q_k$ (see \cite{GLS08}). In particular,
\begin{equation}\label{eq:Psi-minor}
  \Psi_{M_J} = \minor{J}.
\end{equation}
The category $\CM C$ equipped with 
$\Psi$ provides an (additive) categorification of the cluster structure on the 
coordinate ring $\CC[\Gr(k, n)]$ from \cite{Scott}. 

\subsection{The subcategories $\GP B$ of $\CM C$}\label{sec:GPB}

Following \cite{JKS3}, let  $ B$ be an $\ring$-order in $C[t^{-1}]\cong M_{n}(\field)$ such that 
\begin{equation}\label{eq-intro:CinB}
   C\subseteq  B\subseteq  C[t^{-1}], 
\end{equation}
and $B$ is rigid as a $C$-module. Let $\CM B$ be the category of Cohen-Macaulay 
$B$-modules, which is a full subcategory of $\CM C$ through the inclusion $C\subseteq B$. 
Define
\begin{equation}\label{eq:defpgb}
\GP B=\{M\in \CM B \colon \Ext^1_B(M, B)=0\},
\end{equation}
which is a full exact subcategory of $\CM C$. In fact,  following \cite[Prop 2.9]{JKS3}, 
for any $M\in\CM B$ and $N\in \GP B$,
\begin{equation}\label{eq:extCB}
\Ext^1_B(M, N)=\Ext^1_C(M, N).
\end{equation}
So in this context, we can drop the subscripts on the functor $\Ext^1(-, -)$. Note also 
that by construction, $\GP B$ and $\CM B$ are full subcategories of $\CM C$. So we 
also drop the subscripts on the functor $\Hom(-, -)$.  

By  \cite[Thm 2.14]{JKS3}, $B$ is Iwanaga--Gorenstein of injective dimension at 
most~$2$, while $\GP B$ is the category of Gorenstein projective $ B$-modules 
and is Frobenius (stably) $2$-Calabi--Yau. We refer the reader to \cite{CKP, Pr} for a 
special case, where the algebra $B$ and the category $\GP B$ is constructed from a 
dimer algebra.  

Note that  \cite[Prop 2.19]{JKS3} provides a combinatorial characterisation of the 
algebra $B$, as a \emph{necklace algebra}. More precisely, as a $C$-module, 
\begin{equation}\label{eq:BN}
   B=\oplus_{I\in \nlace} M_I,
\end{equation}
where $\nlace$ is a Grassmann necklace (see \S \ref{sec:positrV} for the definition) determined by $B$. 

\subsection{Positroid varieties}\label{sec:positrV}

We recall some basic definitions related to positroids. We refer the reader to 
\cite{OPS, Pos} for further details. 

A sequence $I_1, \dots, I_n$ of $k$-sets in $\labsubset{n}{k}$ is a 
\emph{Grassmann necklace of type} $(k, n)$ if
\begin{equation}\label{eq:neck-cond}
   I_i\setminus \{i\}\subseteq I_{i+1},
\end{equation} 
for all $i\in [n]$. In particular,  $I_{i+1}=I_i$ when $i\not\in I_i$. Here the indices 
of the necklace are cyclically ordered and can be identified with $\ZZ_n$. 
For $I\subset [n]$ and $i\in [n]$, we use the shorthand $Ii$ to 
denote the set $I\cup \{i\}$.

Let $\leq_i$ be the shifted order on $[n]$, 
\[
   i<_i i+1<_i\dots<_i n<_i 1<_i\dots <_i i-1.
\]
Extend $<_i$ to $\labsubset{n}{k}$ as follows. For two $k$-sets 
$I=\{i_1<_i\dots <_i i_k\}$ and $J=\{j_1<_i\dots <_i j_k\}$, define 
\[
    I\leq_i J \text{ if and only if } i_s\leq_ij_s, ~\forall s.
\]
A Grassmann necklace $\nlace=\{I_1, \dots, I_n\}$ defines a \emph{positroid}
\[
   \positr_{\nlace}=\{J\in \labsubset{n}{k}\colon I_i\leq_i J, ~ \forall i\}.
\]
Grassmann necklaces and positroids are both in one-to-one correspondence 
with decorated permutations of type $(k, n)$  \cite{OPS, Pos}.  By \cite[\S 2]{JKS3}, 
there is also a correspondence between these combinatorial objects and the orders $B$. 

Now the positroid and open positroid varieties associated to a necklace $\nlace$ can be defined 
as follows (see \cite[\S 5.4]{KLS}), 
\[
  \Pi=\{V\in \Gr(k, n)\colon I\not\in \positr_\nlace\Longrightarrow  \Delta_I(V)=0\},
\]
and 
\[
   \opoid~ =\{V\in \Pi \colon \Delta_I(V)\not=0 \text{ for all } I \in \nlace\}.
\]
By \cite[Thm 5.9]{KLS}, the varieties $\Pi$ and  $\openpv$ are irreducible of codimension 
$a(\sigma)$, where $\sigma$ is the decorated permutation determined by $\nlace$ and $a(\sigma)$
is the number of alignments determined by $\sigma$ (\cite[Def. 4.6]{OPS} and \cite[p. 67]{Pos}).

Associated to a decorated permutation $\sigma$ of type $(k, n)$, there are two types 
of graphs, namely, reduced Postnikov diagrams and reduced \emph{plabic} (planar bicoloured) 
graphs \cite{Pos}. They are in one-to-one correspondence and can be constructed from each other. 
We will not emphasise the term reduced again,  \emph{all Postnikov diagrams and plabic graphs considered 
from this point onwards are assumed to be reduced}.
Label 
the faces by $k$-sets, following the left-target convention, where the faces are faces of a plabic 
graph or the alternating regions of the corresponding Postnikov diagram.  
Let $\positr$ and $\nlace$ be the positroid and the Grassmann necklace determined by $\sigma$.
Oh--Postnikov--Speyer \cite[Thm. 1.5]{OPS} showed that $\maxWS$ is a maximal collection of non-crossing sets 
 in $\positr$ containing $\nlace$  if and only if 
$\maxWS=\flabs$, the set of the face labels for a plabic graph $G$ of type $\sigma$. Moreover, two 
such maximal collections can be obtained from each other by a sequence of square moves,  
$\maxWS\mapsto (\maxWS\backslash \{Iac\})\cup \{Ibd\}$, where $a, b, c, d\in [n]$ are cyclically ordered and $\maxWS$ contains $Iab, Ibc, Icd, Ida, Iac$ (\cite[Thm 1.4]{OPS}). Following 
\cite[Prop 17.10]{Pos}, they deduced that $ |\flabs|= k(n-k)-a(\sigma)+1$ (\cite[Thm 6.8]{OPS}). 
Therefore,  
\begin{equation}\label{eqPositroidDim}
\dim \Pi=\dim \opoid=|\flabs|.
\end{equation}

\subsection{Cluster structure on open positroid varieties}
Leclerc  \cite{Lec} constructed Frobenius subcategories $\catC_{v, w}$ of the module 
category of a preprojective algebra and proved that those subcategories 
possess a cluster structure. He further showed that the categorical cluster 
structure on $\catC_{v, w}$ determines a cluster structure on the corresponding Richardson 
varieties and conjectured that the cluster algebra constructed from $\catC_{v, w}$  coincides with the coordinate ring of the variety. The conjecture was confirmed 
in type $\mathbb{A}$ by Serhiyenko--Sherman-Bennett \cite{SSB}. 
Note that positroid varieties 
are Richardson varieties (see for instance \cite{KLS}). Muller-Speyer 
\cite{MS} conjectured that the coordinate ring of an open positroid variety 
has a cluster structure determined by the combinatorics of a Postnikov diagram for 
the corresponding positroid. Leclerc proposed a question of comparing 
the two cluster structures \cite{Lec}.
Galashin--Lam \cite{GL} confirmed that the conjecture is true and that 
the two cluster structures coincide.
More recently, for connected positroid varieties, 
Pressland  \cite{Pr} applied a category of the form $\GP B$ to give 
a categorification of Galashin--Lam's cluster structure on the associated positroid variety, 
where $B$ is the boundary algebra of a dimer algebra constructed by Canakci--King--Pressland
\cite{CKP}. 

\subsection{Main results}
We study the cluster structure of $\GP B$, its relationship to the cluster structure 
of $\CM C$ and its connection to the coordinate ring of the associated 
positroid variety.

We begin by proving that $\GP B$ is a functorially finite subcategory of $\CM C$ (Theorem \ref{thm:FF}).
We then exploit an important property of extension closed and functorially finite subcategories of
a triangulated 2-CY category \cite[Prop. II.2.3]{BIRS} to investigate the cluster substructure 
(as defined in \cite{BIRS}) on $\GP B$ relative to $\CM C$.
We construct a subcategory $\catD_2$ of $\CM C$, which is 
$\Ext^1$-perpendicular to $\GP B$. 

Let $\maxNC$ be a maximal collection of non-crossing sets $I$
with $M_I\in \GP B$, and let $\maxNC'$ be 
a collection of sets disjoint to $\maxNC$
such that $\maxNC\cup \maxNC'$ is a maximal collection of 
non-crossing sets in $\labsubset{n}{k}$.
Denote by $Q_U$ the Gabriel quiver of $(\End U)^{\text{op}}$, and by 
$Q_U^\circ$ the subquiver of $Q_U$ obtained by removing the arrows between vertices corresponding to the 
indecomposable projective summands in $U$. The quiver $Q^\circ_U$ is a called 
an \emph{ice Gabriel quiver}.

\begin{theorem}  [Theorem \ref{main1}] 
The category $\GP B$ is a cluster subcategory of $\CM C$. More precisely, we have
the following properties. 
\begin{enumerate}
\item\label{itmclpgb1} A module in $ \GP B$  is a cluster tilting object if and only if it is maximal rigid.
In particular, $M_{\maxNC}$  is a cluster tilting object in $\GP B$. 

\item\label{itmclpgb12} There is $V\in \catD_2$ (e.g. $V = M_{\cmax}$) such that $T=U\oplus V$ is a cluster tilting object in $\CM C$ for any cluster tilting object $U$ in $\GP B$. 
\end{enumerate}

Assuming $T$ is as in (2):
\begin{enumerate} 
\item[(3)] The mutation sequences of $U$ at $U_i$ in $\GP B$ coincide with 
the mutation sequences of $T$ at $U_i$ in $\CM C$.

\item[(4)] $Q_U^\circ$ is a subquiver of $Q_T$ such that all arrows in $Q_T$
incident to a mutable vertex in $Q_U$ belong to $Q_U^\circ$.
\item[(5)] $Q_U^\circ$ has no loops and no 2-cycles.  
Therefore, $Q_{\mu_iU}^\circ=\mu_iQ_U^\circ$. 
\end{enumerate}
\end{theorem}

Theorem \ref{main1} implies that $\GP B$ determines a cluster subalgebra $\clalg$ of $\CC[\Gr(k, n)]$. 
Moreover, we have the following, relating  $\clalg$  to Muller--Speyer's cluster algebra $\algMS$ defined 
by the quiver $Q_G$ constructed from a plabic graph $G$ \cite{MS}.

\begin{theorem}[Theorem \ref{thm:locacy}] \label{main3}
Let $G$ be a plabic graph of type $\sigma$ and $\maxNC=\mathcal{F}(G)$. Then 
 $Q_{M_\maxNC}^{\circ}=Q_G$.
Consequently,  $\clalg\cong \algMS$  and $\clalg$ is locally acyclic.
\end{theorem}

Next, we investigate the relation between the cluster algebras constructed from 
$\GP B$ and the coordinate rings of positroid and open positroid varieties.
A key  intermediate step is to prove the existence of generic 
bases for  certain subalgebras  and for an ideal of $\CC[\Gr(k, n)]$ arising from relevant subcategories of $\CM C$. 
These include $\algGP$, which is spanned by $\Psi_M$ for all $M\in \GP B$,  
 $\clalg$ and another cluster algebra $\clalgptf$ constructed from  
the initial seeds defined by a cluster tilting object of the form $M_\maxNC$ and  
the generalised partition function $\ptf$ in \eqref{eq:defptf}. 
We also consider their localisations $(\algGP)_B$, $\clalgptfloc$ at 
$\{\Delta_I\colon M_I\in \add B\}$ and $\{\Delta^{-1}_I\colon M_I\in \add B\}$, respectively.
Let 
\begin{equation*}
p\colon \CC[\Gr(k, n)]_B \to \CC[\opoid],
\end{equation*}
the localisation of the quotient map $\CC[\Gr(k, n)] \to \CC[\Pi]$ at $\{\Delta_I\colon M_I\in \add B\}$.

\begin{theorem} \label{main2} [Theorem \ref{thm:cat}]  We have the following.
\begin{enumerate}
\item \label{itm:newcat1} $\algGP=\clalg$ and $\clalg\cong \clalgptf$ as cluster algebras.

\item The  map $p$ restricts to an isomorphism from $(\algGP)_B$ to $\CC[\opoid]$ and 
an injection  from $\algGP$ to $ \CC[\Pi] $.

\item The  map $p$ restricts to an isomorphism from $(\algGPptf)_{B^{-1}}$ to $\CC[\opoid]$ and 
an injection  from $\algGPptf$ to $ \CC[\opoid] $.
\end{enumerate}
\end{theorem}

As a consequence, we obtain a new proof Galashin--Lam's Theorem, which confirms 
Muller--Speyer's conjecture  \cite[Conj. 3.4]{MS} (see also Conjecture \ref{conj:MS}). 
Galashin--Lam~lifted Leclerc's cluster structure from \cite{Lec} and showed that the lifted cluster algebra is 
isomorphic to the cluster algebra  $\algMS$ constructed by Muller--Speyer \cite{MS}. 
To establish a key step proving their theorem--namely, the coordinate ring of the open positroid 
variety is contained in the upper cluster algebra of $\algMS$--they made use of techniques from
combinatorics of Le-diagram, 
parametrisations of open Deodhar strata \cite{MR04},
Postnikov's boundary measurement map \cite{Pos} and 
Muller-Speyer's twist \cite{MS17} for positroid varieties. 

Evidently, our approach is not only technically different from that of Galashin--Lam, but also 
arises from a distinct perspective on understanding the cluster structure of $\clalg$. In particular, we do not rely on Leclerc's cluster structure, 
and to our knowledge,  no direct link between $\GP B$ and Leclerc's categories has been established.
We emphasise the importance of the properties of $\GP B$ as  the cluster subcategory of $\CM C$, together with  
Grassmannian combinatorics, which
are crucial for  proving Theorem~\ref{main3}.  
To prove Theorem~\ref{main2}, we employ properties of cluster characters to construct generic bases, 
which leads to a surjective map,
\[
\CC[\opoid] \to (\clalg)_B.
\]  
Finally, the surjective map is an isomorphism, because both algebras 
are coordinate rings of irreducible varieties of the same dimension. 

We remark that the second author lifted Leclerc's categories modelling open positroid varietes to $\CM C$ in his thesis \cite{Rio}. The categories obtained are in general different from $\GP B$ and offer a distinct cluster structure.
In \S \ref{sec:ex}, we include examples of the construction from \cite{Rio}.
We also discuss generators of the coordinate rings of open positroid varieties. By definition and Corollary~\ref{cor:rk1oid},  such a coordinate ring is generated by $\Delta_I$,  $\Delta^{-1}_J$ for 
$M_I\in \CM B$ and $M_J\in \add B$.  
In the case of positroids of type $(2, n)$, the result can be enhanced. That is, the ring is generated by $\Delta_I$,  $\Delta^{-1}_J$ for 
$M_I\in \GP  B$ and $M_J\in \add B$ (Proposition \ref{rem:gr2n}). 

\section{Functorial finiteness of $\GP B$}\label{Sec:Func}

Throughout this paper, all categories are Krull-Schmidt and defined over 
the complex numbers. A category $\catC$ is \emph{Frobenius (stably) 2-CY} 
if it is exact Frobenius and the stable category is a 2-CY triangulated category with finite
dimensional homomorphism spaces. All Frobenius categories $\catC$ will be 
extension closed and full (i.e. full exact) subcategories of either $\CM C$ or $\sub Q_k$. 
This means objects and homomorphism spaces have finite dimension 
(in the latter case) or finite rank (in the former), and also that these categories 
have finitely many indecomposable projective-injective objects.

Let $\catD$ be a full additive subcategory of $\catC$ and $X\in \catC$. Recall that 
$f\colon Y\to X$ for some $Y\in \catD$ is called a {\it right $\catD$-approximation} 
of $X$ if 
\[
\Hom(Z, f)\colon \Hom(Z, Y) \to \Hom(Z, X)
\]
is surjective for any $Z\in \catD$. 
Similarly,  $f\colon X\to Y$ for 
some $Y\in \catD$ is called a {\it left $\catD$-approximation} of $X$ if 
\[
\Hom(f, Z)\colon \Hom(X, Z) \to \Hom(Y, Z)
\]
is surjective for any $Z\in \catD$. Note the feature of $\CM C$  that any 
left approximation $f\colon X\to Y$ is injective. Indeed, $P_0$ embeds into any 
$M\in \catD ~(\subseteq \CM C)$, and every projective $P_i$ embeds into $P_0$. So there is at 
least one injection from $X$ to an object in $\catD$, and hence the approximation 
$X\rightarrow Y$ is injective. 

A right approximation $f\colon Y\to X$ is said to be \emph{minimal} if $f=fg$ for some 
map $g\colon Y\to Y$ implies that $g$ is an isomorphism. Any other right 
approximation $f'\colon Y'\to X$ can be decomposed as 
$f'=(f_1, f_2): Y_1\oplus Y_2\to X$ such that $f_2=0$ and $f_1$ is minimal. 
Similarly,  a left approximation $f\colon X\to Y$ is \emph{minimal} if $f=gf$ for some $g\colon Y\to Y$ implies that $g$ is an isomorphism. In this case,  any other left 
approximation $f'\colon X\to Y'$ can be decomposed as 
$f'=(f_1, f_2): X\to Y_1\oplus Y_2$ such that $f_2=0$ and $f_1$ is minimal. 
Note that minimal approximations exists and are unique, up to isomorphism. We refer to 
\cite[Sec. I.2]{ARS} for proof for finite length categories. The proofs 
are similar when objects have finite rank.{\it We will omit 'left' and 'right' when 
the context makes it clear which approximation is being referred to}. We say the subcategory $\catD$  is \emph{contravariantly (resp. covariantly)  finite} if any object in $\catC$ has a 
right (resp. left) $\catD$-approximation, and  
\emph{functorially finite} if it is both contravariantly and covariantly finite.

The following is the main result of this section.

\begin{theorem}\label{thm:FF}
The category $\GP B$ is a functorially finite, extension closed Frobenius 2-CY 
subcategory of $\CM  C$.
\end{theorem}

Following \cite[Thm 2.14]{JKS3}, we know that $\GP B$ is Frobenius 2-CY, 
full and extension closed. So it remains to prove its functorial finiteness, which will
follow from Proposition \ref{prop:leftapp} and Proposition \ref{prop:rightapp}.

\subsection{Perpendicular subcategories}

Let $\catC$ be a Frobenius (stably) 2-CY category as above, with a projective-injective 
generator $P$ and denote by $\ul \catC$ its stable category with  morphisms
\[
\ul \Hom(X, Y)=\Hom(X, Y)/\Hom(X, Y)_P
\]
for any $X, Y\in \catC$,  where $\Hom(X, Y)_P$ is the subspace of homomorphisms 
that factor through $\add P$.

\begin{lemma}\label{lem:funfst}
Let $\catB$ be a subcategory of $\catC$ containing $P$. Then $\catB$ is functorially 
finite if and only if the quotient $\ul \catB$ is functorially finite in $\ul \catC$.
\end{lemma}
\begin{proof}
By definition, the functorial finiteness of $\catB$ in $\catC$ implies the functorial finiteness of $\ul \catB$ in $\ul\catC$.
We prove the converse. For any $X\in \ul\catC$, there exist $M\in \ul\catB$ and $f\colon M\to X$ in $\ul\catC$ such that any map $g\colon M'\to X$ with $M'\in \ul\catB$, we have 
$g=fh$ for some $h\colon M'\to M$. Lift the maps to $\catC$, we have 
$\widehat{g}-\widehat{f}\widehat{h}\in \Hom(M', X)_P$, that is, 
$\widehat{g}-\widehat{f}\widehat{h}=ab$ for some $P'\in \add P$, $a\colon P'\to X$ and $b\colon M'\to P'$.
Let $d\colon P''\to X$ be an $\add P$-approximation of $X$. Then there exists $c\colon P'\to P''$ 
such that $a=dc$. Consequently, $\widehat{g}=\widehat{f}\widehat{h}+dcb$. That is, 
$(\widehat{f}, d)\colon M\oplus P''\to X$ is a right $\catB$-approximation of $X$. So $\catB$ is contravariantly 
finite, and similarly, it is covariantly finite.
\end{proof}

\begin{lemma} \label{prop:perp}
Let $X$ be a rigid object in $\catC$ and let $\catB=\{M\in \catC\colon \Ext^1(M, X)=0\}$.
Then $\catB$ is functorially finite in $\catC$.
\end{lemma}
\begin{proof}
Note that $\add X$ is functorially finite and extension closed in $\ul\catC$, and 
\[\ul\catB=\{M\in \ul \catC\colon \ul\Hom(M, X[1])=0\},\] 
which is functorially finite in $\ul\catC$ by  \cite[Thm II.2.1]{BIRS}. 
Therefore, following Lemma \ref{lem:funfst}, $\catB$ is functorially finite in $\catC$.
\end{proof}

Note that the functorial finiteness of the category $\catB$ in Lemma \ref{prop:perp} can also 
be checked directly following the definition.
As $\catB$ is a lift of the perpendicular subcategory $\ul\catB={}^{\perp}X[1]$ in $\ul\catC$, we say $\catB$ is  the \emph{perpendicular subcategory} of $X$ in $\catC$ (with respect to $\Ext^1(-, X)$). 

\subsection{Existence of right approximations}

\begin{lemma} \cite[Lem 2.12]{JKS3}\label{lem:GPBBapp} 
Let $K$ be the kernel of an $\add B$-approximation of some $M$ in $\CM C$.  
Then $K\in\GP B$. In particular, if $M\in \CM B$, then any syzygy $\Omega_B M\in\GP B$.
\end{lemma}

\begin{proposition}\label{prop:leftapp}
Let $X\in \CM C$. There exists a $\GP B$-approximation $f\colon M\longrightarrow X$. 
\end{proposition}

\begin{proof}
Let $f\colon B'\to X$ be an $\add B$-approximation.
Let $Y=\im f$. It follows that 
$Y\in \CM B$  and $Y$ is the maximal submodule of $X$ that is also a $B$-module. Therefore the embedding $\iota\colon Y\to X$ is a $\CM B$-approximation of $X$.

Consider the following short exact sequence 
\[
\xymatrix{
0\ar[r] & \Omega \ar[r]^{d'} & B'  \ar[r]^f&  Y\ar[r]&0,
}
\]
where $B'\in \add B$.
By Lemma \ref{lem:GPBBapp},  $\Omega\in \GP B$. 
Extend $d'$ to an $\add B$-approximation 
$$d=\begin{pmatrix}d' \\ d''\end{pmatrix}\colon\Omega\to B'\oplus B''.$$
We have the commutative diagram,  
\begin{equation}\label{eq:gpBappofcmB}
\begin{split}
\xymatrix{
0\ar[r]& \Omega\ar[r]^{d ~~~~~~\;\;\;\;}\ar@{=}[d] &B'\oplus B''\ar[r]\ar[d]^b & Z\ar[r]\ar[d]^c &0\\
0\ar[r]& \Omega\ar[r]^{d'}  &B'\ar[r]^f  & Y\ar[r]
 &0,
}
\end{split}
\end{equation}
where  $Z=\cok d$ and $b$ is the projection onto $B'$. 
So all the vertical maps are surjective, and 
by the Snake Lemma, we have the exact sequence 
\begin{equation}\label{eq1}
\xymatrix{
0\ar[r] &B''\ar[r] &Z \ar[r]^c &Y\ar[r]& 0.
}
\end{equation}
Consequently, $Z\in \CM B$ by \eqref{eq:extCB}. Applying $\Hom(-, B)$,  we have 
\[
\xymatrix{
0\ar[r]& \Hom(Z, B)\ar[r] &\Hom(B'\oplus B'', B)\ar[r] &  \Hom(\Omega, B)\ar[r]
&\Ext^1(Z, B)\ar[r] &0.
}
\]
As $d$ is an $\add B$-approximation, $\Hom(d, B)$ is surjective. Therefore 
\[\Ext^1(Z, B)=0,\] and so by the definition of $\GP B$ (see \eqref{eq:defpgb}), we have $Z\in \GP B$. 
Furthermore, applying  $\Hom(L,- )$ with $L\in \GP B$ to \eqref{eq1} gives a 
short exact sequence, and so $c$ is a $\GP B$-approximation. 

Now consider the composition  $\alpha=\iota c$.  
Then $\alpha$ is a $\GP B$-approximation of $X$. Indeed, let $L\in \GP B$ and 
$g\colon L\to X$. 
Then $g$ factors through $\iota$, that is, 
\[g= \iota h\]
for some $h\colon L\to Y$. 
As $c$ is a $\GP B$-approximation,   $h$ factors through $c$. 
Consequently, $g$ factors through $\alpha$. This completes the proof.
\end{proof}

\begin{corollary} \label{cor:cmbff} We have the following.
\begin{enumerate}  
\item\label{it:cor1} The category $\CM B$ is contravariantly finite in $\CM C$.
\item\label{it:cor2} Any $Y\in \CM B$  has a $\GP B$-approximation in an exact sequence 
as follows, 
\begin{equation}\label{eq:appseqofcmb}
\xymatrix{
0\ar[r] & B'' \ar[r]^{d'} & Z  \ar[r]^c&  Y\ar[r]&0,
}
\end{equation}
with  $B''\in \add B$.
\end{enumerate}
\end{corollary}

\begin{proof}
(1) is the first part of the proof of Proposition \ref{prop:leftapp}.
Now for any $Y\in \CM B$, using the construction in  \eqref{eq:gpBappofcmB}, 
the map $c$ is a $\GP B$-approximation of $Y$ 
with $\ker c=\ker b=B''\in \add B$. So (2)  holds. 
\end{proof}

\subsection{Existence of left approximations}

Let $M\in \CM C$ and let $f\colon M\to B'$ be an (injective) $\add B$-approximation 
with cokernel $X$. Consider the  commutative diagram, 
\[
\xymatrix{
0\ar[r] & M \ar[r]^f\ar[d]^a & B'\ar@{=}[d] \ar[r]^{f'} &  X\ar[r] \ar[d]^c&0\\
0\ar[r] & \tM \ar[r]^{\iota} &  B' \ar[r]^{\iota'} & X/tX \ar[r] & 0, 
}
\]
where $tX$ is the maximal torsion submodule of $X$ and so  $X/tX$ is torsion free. 
\begin{proposition}\label{prop:rightapp}
We have $\tM\in \GP B$ and the map $a$ is a $\GP B$-approximation of~$M$.
\end{proposition}
\begin{proof}
First note that $X/tX\in \CM B$, since it is torsion free (and so free over $\ring$) and is a factor module of $B$. 
By construction, $B'$ is projective and $\iota'$ is surjective, and so $\iota'$ is an $\add B$-approximation. Now,
 by Lemma \ref{lem:GPBBapp}, we have $\tM\in \GP B$.

Now let $N\in \GP B$ and $g\colon M\ra N$. 
We will show that $g$ factors through $a$ and thus 
conclude that the map $a$ is a $\GP B$-approximation of $M$.

Since $\GP B$ is Frobenius, there is an $\add B$-approximation $d:N\rightarrow B''$
and a short exact sequence as follows
\begin{equation*}
\xymatrix{0\ar[r]& N\ar[r]^d&  B'' \ar[r]^{d'}& Z\ar[r]& 0,}
\end{equation*}
with $Z\in \GP B$. 

As $f\colon M\to B'$ is an $\add B$-approximation, the map $dg$ factors through $f$, namely
\[
dg=hf
\]
for some $h\colon B'\to B''$. Therefore, there exists $l\colon X\to Z$ such that
$lf'=d'h$ and we have a commutative diagram,
\[
\xymatrix{
0\ar[r] & M \ar[r]^f\ar[d]^a \ar@/_1pc/[dd]_{\begin{smallmatrix}g\\ \\ \\ \\ \\\end{smallmatrix}}& 
B'\ar[d]^{\verteq} \ar[r]^{f'} \ar@/_1pc/[dd]_{\begin{smallmatrix}h\\ \\ \\ \\ \\ \end{smallmatrix}}&  
X \ar@/_1pc/[dd]_{\begin{smallmatrix}l\\ \\ \\ \\ \\ \end{smallmatrix}} \ar[r] \ar[d]^c&0\\
0\ar[r] & \tM 
\ar[r]^{\iota}&  B' \ar[d]^h \ar[r]^{\iota'\; \;}  & X/tX \ar[r] 
& 0\\
 0\ar[r] & N \ar[r]^d&  B'' \ar[r]^{d'}  & Z \ar[r] & 0.
}
\]
Now as $Z$ is torsion free, by applying $\Hom(-, Z)$ to the short exact sequence 
\begin{equation*}\xymatrix{0\ar[r] & tX\ar[r] & X\ar[r] &X/tX\ar[r] & 0,}\end{equation*} 
we have  \[\Hom(X, Z)=\Hom(X/tX, Z). \]
Therefore, there exists some $p\colon X/tX\to Z$ such that $l=pc$. 
We have 
\[
p\iota'=pcf'=lf'=d'h.
\]
Hence, there exists $b \colon \tM \to N$ such that $h\iota=db$. Consequently, 
\[
dg=hf=h \iota a =d b a.
\]
As $d$ is injective, we have \[g=b a,\]  as required. 
\end{proof}

\begin{remark} \label{rem:cmbfunct}
Note that $\CM B=\sub B\cap \fac B$. 
The proofs of Proposition \ref{prop:rightapp} and Proposition \ref{prop:leftapp} imply 
that $\CM B$ is also functorially finite in $\CM C$. This property is similar to a result
for Artin algebras, that is, under some assumptions, the subcategory 
$\sub M\cap \fac N$ is functorially finite in the module category of finite length modules 
\cite[Thm 5.10]{AS}. 
\end{remark}

\section{Cluster subcategories of $\CM C$}
Let $\catC$ be a Frobenius 2-CY category as in Section \ref{Sec:Func}. Recall
that a \emph{cluster tilting object} in $\catC$ is an object $T\in \catC$ such
that $\Ext^1(T, X)=0$ implies that $X\in \add T$. Cluster tilting objects are similarly 
defined for triangulated categories (e.g. the stable category of $\catC$), and for 
full extension closed subcategories. 

Let $Q_T$ be the Gabriel quiver of $A_T=(\End T)\op$, which is the quiver of the 
\emph{basic} algebra Morita equivalent to $(\End T)\op$. Let $Q_T^\circ$ be 
the subquiver of $Q_T$ obtained by removing the arrows between vertices corresponding to the 
indecomposable projective summands in $T$. 
The Gabriel quiver encodes the dimension of the space of irreducible maps 
between indecomposable summands of $T$. 

We call $\catC$  a \emph{cluster category} if it admits a \emph{cluster structure} 
\cite{BIRS} (see also \cite{KR1}) consisting of a collection $\catC_0$ of cluster
tilting objects. In particular, the Gabriel quiver $Q_T^\circ$ for $T\in \catC_0$ has no
loops and no two-cycles, and $Q_{\mu_i T}^\circ=\mu_iQ_T^\circ$, 
where $\mu_iT$ is the mutation of $T$ at the summand $T_i$ and $\mu_iQ_T^\circ$ 
is the mutation of $Q_T$ at the vertex $i$, defined by Fomin-Zelevinsky \cite{FZ1}. 
Recall that mutation of cluster tilting objects is defined using (exact) mutation sequences, 
\begin{equation} \label{eq:mutsequences}
0\lra T_i^* \lra E \lra T_i \lra 0 \text{ and } 0\lra T_i \lra F \lra T_i^* \lra 0,
\end{equation}
where the maps are $\add T\backslash T_i$-approximations (see \cite{BIRS}),
and $\mu_iT = (T\backslash T_i) \oplus T_i^*$.

When $\catC_0$ consist of all the cluster tilting objects which can be obtained by 
iterated mutation from a fixed cluster tilting object $T$, we say the 
\emph{cluster structure is defined} by $T$. Cluster tilting objects in $\catC_0$ 
are then said to be \emph{reachable} from $T$. We drop $T$ and say that a 
cluster tilting object is \emph{reachable}, when $T$ is clear from the context.

Let $\catB$ be a full exact Frobenius subcategory of $\catC$. Assume that $\catB$ and 
$\catC$ are cluster categories. The cluster structure on  $\catB$ is a {\it cluster 
substructure of $\catC$} \cite[Sect II.2]{BIRS} if there exists  an object $X\in \catC$ 
such that $T'\oplus X$ is a cluster tilting object in $\catC_0$ for any cluster tilting 
object $T'\in \catB_0$. In this case, we call $\catB$ a {\it cluster subcategory} of 
$\catC$.

\subsection{Two results of Buan-Iyama-Reiten-Scott} 

We recall two results on cluster tilting objects and cluster structures from \cite{BIRS},
which we will apply to show that $\GP B$ is a cluster subcategory of $\CM C$. 

\begin{theorem} \cite[Thm II.1.8]{BIRS} \label{BIRS2}
Let $\catC$ be a Frobenius 2-CY category with a cluster tilting object $T$. Then 
any maximal rigid object in $\catC$ is a cluster tilting object. 
\end{theorem}

Note that by definition, cluster tilting objects are maximal rigid. So by 
Theorem \ref{BIRS2}, cluster tilting objects and maximal rigid objects coincide in 
the category $\catC$. The following result is \cite[Prop II.2.3]{BIRS} applied to the
stable category $\ul{\CM C}$.

\begin{proposition} \label{prop:tem}  \cite[Prop II.2.3]{BIRS}
Let $\catD'$ a functorially finite and extension closed subcategory of $\ul{\CM C}$. Let 
\[
\catP=\catD'\cap {}^{\perp}\catD'[1]
\text{ and } \catD={^\perp\catP[1]}.
\] 
\begin{enumerate}
\item 
There exists a functorially finite and extension closed subcategory $\catD''$ of $\ul{\CM C}$ that contains $\catP$ such that  as triangulated categories,
\begin{equation*}
\catD/\catP \cong \catD'/\catP \times \catD''/\catP.
\end{equation*}
In particular, for any $M'\in \catD'$ and $M''\in \catD''$, any morphism 
between $M'$ and $M''$ factors through $\catP$.
\item  The map $(T', T'')\mapsto T'\oplus T''$
is a bijection (up to isomorphism and repeated summands)  
between pairs of cluster tilting objects  $T'\in \catD'$ and $T''\in 
\catD''$,  and cluster tilting objects in $\catD$ that contain $\catP$.
\end{enumerate}
\end{proposition}

\subsection{The cluster substructure on $\GP B$}

Let $\iota: C\to B$ denote the inclusion of orders $C\subseteq B$. 
Note that $\iota$ is also a homomorphism of $C$-modules.

\begin{proposition} \label{Prop:CBappr}
The embedding $\iota\colon C\to B$  is a $\CM B$-approximation of $C$. Thus,  
it is also a $\GP B$-approximation.
\end{proposition}
\begin{proof}
Let $M\in \CM B$ and let $f\colon C \rightarrow M$ be a $C$-homomorphism. Then
$f(c)=cf(1_C)$. Define $g\colon B\rightarrow M$ by $g(1_B)=f(1_C)$. Then
since, $\iota(1_C)=1_B$, we have $f=g\iota$. 
\end{proof}

As a corollary, we have the following combinatorial 
characterisation of rank one modules in $\CM B$.

\begin{corollary}\label{cor:rk1oid} Suppose that $B=\oplus_i M_{I_i}$ as a $C$-module.
Let $M_I$ be a rank one $C$-module. Then $M_I\in \CM B$ if and only if $I$ is in the positroid $\positr$ associated to $B$, that is, $I_i\leq_i I$ for all $i$. 
\end{corollary}
\begin{proof}
We can write $\iota=\oplus_j \iota_j$ with $\iota_j\colon Ce_j\to M_{I_j}$. 
Then for any $I$ in the positroid, $I_j\leq_j I$ for all $j$. By comparing the 
profiles of the modules (see \cite{JKS1} and also Example \ref{Ex:newEx} for an illustration), 
we have an embedding $f_j\colon M_{I_j} \to M_I$
 that is the identity at vertex $j$, for any $j$. Therefore, 
 $(f_j)\colon B\to M_I$ is a surjective map and so $M_I\in \CM B$. 
 
 Conversely, suppose that $M_I\in \CM B$. Let $g_j\colon  C e_j\to M_I$ such that it
 is an isomorphism at vertex $j$. By  Proposition \ref{Prop:CBappr}, 
$g_j$ factors through  $\iota_j$, that is, 
 \[
 g_j=h\iota_j
 \]
for some $h\colon M_{I_j}\to M_I$. 
Since $\iota_j$ and $g_j$ are 
isomorphisms at vertex $j$,  $h$ is an injective map 
and is an isomorphism at vertex $j$. That is, there is an embedding 
$M_{I_j}\to M_I$ that identifies the spaces at vertex $j$. Again by comparing the 
profiles of the modules, we have  $I_j\leq_{j+1} I$ for all $j$.
Consequently, $I$ is contained in the positroid $\positr$.
\end{proof}

\begin{example}\label{Ex:newEx}
In this example, we draw the contours of three rank one modules $M_{4569}$, $M_{2378}$ and $M_{1457}$ 
for $(k, n)=(4, 9)$ within the $4 \times 5$ rectangle. The 
contour of each module is a sequence of down or up steps with the down steps 
are labeled by the profile, starting from the left corner, which is vertex $9$ (or 0) and 
is identified with the right corner.  For instance, the green line is the contour for 
module $M_{4569}$ with steps 4, 5, 6, 9 downwards. Note that the module structure 
of each rank one module $M_I$ is completely determined by the contour.

The embedding $M_{1457}$ into $M_{4569}$ as depicted in the picture identifies the spaces 
of the two modules at vertex 9, while to embed $M_{2378}$ into $M_{4569}$, we need to 
shift $M_{2378}$ one step down and so the embedding does not identify the spaces of the two 
modules at vertex 9.
Therefore, $1457\leq_1 4569$, but $2378\not\leq_1 4569$. 
\begin{figure}[h]
\begin{tikzpicture} [scale=0.4,
upbdry/.style={thick, gray},
lowbdry/.style={thick, gray},
bboxes/.style={thick,  densely dotted},
rboxes/.style={thick,  densely dotted},
ridge/.style={very thick, purple},
bridge/.style={very thick, blue},
gridge/.style={very thick, green},
>={Stealth[inset=2.5pt,length=4.5pt,angle'=40,round]},
outarr/.style={->, gray},
midarr/.style={->,blue},
rdgarr/.style={->,red},
keyarr/.style={->,black},
outdot/.style={gray, fill= gray},
middot/.style={blue, fill=blue}]
\newcommand{\abit}{0.133}
\begin{scope} [scale=1.2, shift={(-10,-5)},rotate=135]
\draw [bboxes] (0,-1)--(3,-1) (0,-2)--(1,-2)--(1,0) (2,0)--(2,-2);
\draw [rboxes] (4,-1)--(3,-1) (4,-2)--(3,-2)--(3,-5) (4,-3)--(1,-3)--(1,-5) (4,-4)--(0,-4) (2,-2)--(2,-5);
\draw [upbdry] (0,-5)--(4,-5)--(4,0);
\draw [lowbdry] (0,-5)--(0,0)--(4,0);
\draw [ridge] (4,0)--(3,0)--++(0,-2)--++(-2,0)--++(0,-1)--++(-1,0)--(0,-5);

\draw [bridge] (3.85,0)--(3.85,-1)--(2,-1)--(2,-4)--(0.15,-4)--(0.15,-5);
\draw [gridge] (4,0)--(4,-3)--(1,-3)--(1,-5)--(0,-5);

\end{scope}
\end{tikzpicture}
\caption{Profiles and contours of modules: $M_{\color{green}4569}$,
$M_{{\color{blue}2378}}$ and $M_{\color{purple} 1457}$}
\label{fig:profile}
\end{figure}
\end{example}

If $\catC$ is a subcategory 
of $\CM C$, denote by $\ul \catC$ the image of $\catC$ under the quotient functor
$\CM C\rightarrow \ul{\CM C}$. 

\begin{proposition}\label{stabBC} The quotient functor $\GP B\to \ul{\GP B}$ induces 
an equivalence between 
$\GP B/\add B$ and $\ul{\GP B}/\ul{\add B}$.
\end{proposition}
\begin{proof} 
Note that if some projective $C$-module $P$ is  in $\GP B$, then 
it is projective in $\GP B$, namely, $P\in \add B$. 
To show that the two categories are equivalent, it suffices to show that for any $X, Y\in \GP B$, if 
$f\colon X\to Y$ factors 
through $\add C$, then it factors through $\add B$. 
Suppose that $f=gh$, where $h\colon X\to C'$, $g\colon C'\to Y$ and $C'\in \add C$.
By Proposition  \ref{Prop:CBappr}, the embedding $\iota\colon C \to B$ is a $\GP B$-approximation.
Hence $g$ factors through  $\add B$. So does $f$, as required. 
\end{proof}

Define
\begin{equation} \label{eq:catD}
\catD=\{M\in \CM C\colon \Ext^1(M, B)=0\}.
\end{equation}
Note that $C\in \catD$ and $\GP B\subseteq \catD$. Let $\catD'=\catD/\add(B\oplus C)$, and 
\begin{equation} \label{eq:catD2}
 \catD_2=\{N\in \catD \colon \Hom_{\catD'}(M, N)=\Hom_{\catD'}(N, M)=0, 
 \text{for all } M\in \GP B\}.
\end{equation}
The following lemma is an easy consequence of the definition of $\catD_2$.

\begin{lemma} We have 
$B\oplus C\in \catD_2$. Consequently, if $V$ is a maximal rigid module in $\catD_2$, then
$B\oplus C\in \add V$.
\end{lemma}

\begin{proposition} \label{prop:GPBV}
For any $M\in \GP B$ and $N\in \catD_2$,  $\Ext^1(M, N)=0$. 
\end{proposition} 
\begin{proof}
Let $r\colon B'\to M$ be an $\add B$-approximation. We have the following short exact sequence
\[\xymatrix{
0\ar[r] &\Omega_M \ar[r]^l & B' \ar[r]^{r}& M \ar[r] & 0.
}\]
Then $\Omega_M\in \GP B$ (see Lemma \ref{lem:GPBBapp}).
Let $f\colon \Omega_M\to N$.
By \eqref{eq:catD2}, $f$ factors through some $X\in \add (B\oplus C)$. We have $f=hg$ with 
\[\xymatrix{\Omega_M\ar[r]^{~~~~~~g} & X \ar[r]^h & N.}\]
As $\Ext^1(M, B\oplus C)=0$, we have $\Ext^1(M, X)=0$. Hence  $g$ factors through $l$ and so does $f$. 
Consequently, $\Ext^1(M, N)=0.$
\end{proof}

\begin{lemma}\label{lem:elabtem}
Using the notation from Proposition \ref{prop:tem}. Let 
$\catD' = \ul{\GP B}$, then 
$$\catP = \ul{\add B} \text{ and } \catD'' = \ul{\catD_2}.$$
\end{lemma}
\begin{proof}
We have 
\begin{eqnarray*}
\ul{\catD'}\cap {}^{\perp}\ul{\catD'}[1]&=&\ul{\{{X}\colon X\in \GP B \text{ and }  
\Ext^1_C(X, \catD')=0\}}\\
&=&\ul{\{{X}\colon X\in \GP B \text{ and }  \Ext^1_B(X, \catD')=0\} } \\
&=&\ul{\add B},
\end{eqnarray*}
where the second equality follows from \eqref{eq:extCB}. 

By the decomposition in Proposition \ref{prop:tem} (1), $\catD''$ consist
of those $M\in \ul{\CM C}$ where all maps to and from 
$\catD' = \ul{\GP B}$ factor through $\ul{\add B}$. Therefore 
$\catD'' = \ul{\catD_2}$.
\end{proof}

\begin{definition}\label{def:NC} \cite{JKS1, OPS} two subsets $I, J$ of $[n]$ are said to be \emph{non-crossing}
if there do not exist pairs of elements $\{a, c\}\in I\backslash J$ and 
$\{b, d\}\in J\backslash I$ such that $a, b, c, d$ are cyclically ordered. 
\end{definition}

Let $\nlace$ be the Grassmann necklace associated to $B$. 
Let $\maxNC$ be any  maximal collection of non-crossing sets in the positroid $\positr_{\nlace}$ 
containing $\nlace$ and 
$\cmax$  a collection of sets disjoint to $\maxNC$ such 
that $\maxNC\cup \cmax$ is a maximal collection of non-crossing $k$-sets, 
which implies that $M_{\maxNC\cup \maxNC'}$ is a cluster tilting object in $\CM C$ 
(see \S \ref{sec:catgr} or \cite{JKS1}). 

\begin{theorem} \label{main1} 
The category $\GP B$ is a cluster subcategory of $\CM C$. More precisely, we have
the following properties. 
\begin{enumerate}
\item\label{itmclpgb1} A module in $ \GP B$  is a cluster tilting object if and only if it is maximal rigid.
In particular, $M_{\maxNC}$  is a cluster tilting object in $\GP B$. 

\item\label{itmclpgb12} There is $V\in \catD_2$ (e.g. $V = M_{\cmax}$) such that $T=U\oplus V$ is a cluster tilting object in $\CM C$ for any cluster tilting object $U$ in $\GP B$. 
\end{enumerate}

Assuming $T$ is as in (2):
\begin{enumerate} 
\item[(3)] The mutation sequences of $U$ at $U_i$ in $\GP B$ coincide with 
the mutation sequences of $T$ at $U_i$ in $\CM C$.

\item[(4)] $Q_U^\circ$ is a subquiver of $Q_T$ such that all arrows in $Q_T$
incident to a mutable vertex in $Q_U$ belong to $Q_U^\circ$.
\item[(5)] $Q_U^\circ$ has no loops and no two-cycles.  
Therefore, $Q_{\mu_iU}^\circ=\mu_iQ_U^\circ$. 
\end{enumerate}
\end{theorem}
\begin{proof}
Note that (1), (3) and (5) imply that $\GP B$ is a cluster category, which together with 
(2) implies that it is a cluster subcategory of $\CM C$. Below we  prove  
(1)-(5).

We first work with the stable category $\ul{\CM C}$. Recall the categories $\catD$, 
$\catD_2$ from \eqref{eq:catD} and \eqref{eq:catD2}. Let $\catD_1=\GP B$.
By Proposition \ref{prop:tem} (and Lemma \ref{lem:elabtem}), 
$\ul{\catD_2}$ is a functorially finite and extension closed subcategory of 
$\ul\catD$ containing $\catP=\ul{\add B}$, and 
\begin{equation}\label{eq6}
\ul\catD/\ul{\add B} \cong (\ul{\catD_1}/\ul{\add B}) \times (\ul{\catD_2}/\ul{\add B}).
\end{equation} 

By assumption, $\maxNC$ contains the necklace $\nlace$. So $B\in \add M_\maxNC$.  
Furthermore, 
$M_{\maxNC\cup \cmax}=M_\maxNC\oplus M_{\cmax  }$ is 
a cluster tilting object in $\CM C$. So
\[
\Ext^1(M_\maxNC\oplus M_{\cmax}, B)=0,
\]
and thus ${M_\maxNC\oplus M_{\cmax}}\in \catD$. 
By Proposition \ref{prop:GPBV}, any $M$ in $\GP B$ with $\Ext^1(M, M_\maxNC)=0$ has no
non-trivial extensions with $M_\maxNC\oplus M_{\cmax}$, and so $M_\maxNC$ is  a cluster
tilting object in $\GP B$. The rest of (1) follows from Theorem \ref{BIRS2}.

By construction, $M_{\cmax}$ has no direct summands in $\GP B$. Therefore,
$M_{\cmax} \in \catD_2$ and $M_{\maxNC'}\oplus B$ is a cluster tilting object in $\catD_2$, by 
Proposition \ref{prop:GPBV}, since any $M$ in $\catD_2$ without extensions with 
$M_{\cmax}$ also has no extensions with $M_\maxNC\oplus M_{\cmax}$.

Let $V$ be a cluster tilting object in $\GP B$. Any indecomposable $M\in \CM C$ 
without extensions with $V\oplus M_{\cmax}$ is in $\GP B$ or $\catD_2$ by \eqref{eq6}, 
and therefore $M\in \add V$ or $M\in \add M_{\cmax}$, which shows that $V\oplus M_{\cmax}$
is a cluster tilting object. This proves (2).

{\it For the remainder of the proof, we let $T=U\oplus V$ be a cluster tilting object 
in $\CM C$ as described in  (2)}.

Let $U_i\in \add U\backslash \add B$. Then $U_i\not\in \add C$ and so is a 
mutable summand of $T$.
Consider the following mutation sequence of $T$ at  $U_i$ in $\CM C$,
\begin{equation}\label{eq:mut}
\xymatrix{
0\ar[r] & U_i^*\ar[r]& E_i\ar[r]^f &U_i\ar[r] &0, 
}
\end{equation}
where  $f$ is a minimal $\add T'$-approximation with $T'= U'\oplus V$
and $U'=\oplus_{j\not= i} U_j$.
In particular, \eqref{eq:mut} is non-split and $U_i^*\oplus U'$ is rigid, 
which implies that $\Ext^1(U_i^*, B)=0$ and thus $U_i^* \in \catD$. Furthermore, by
Proposition \ref{prop:GPBV}, $U^*_i\in \GP B$ and so $E\in \GP B$. 
Hence $f$ is an $\add U'$-approximation. Similarly, the other mutation sequence of $T$ at $U_i$ in 
$\CM C$
\begin{equation}\label{eq:mut2}
\xymatrix{
0\ar[r] & U_i\ar[r]^g& F_i\ar[r] &U_i^*\ar[r] &0, 
}
\end{equation}
is also an exact sequence in $\GP B$ and $g$ is a minimal $\add U'$-approximation.
So mutation for a non-projective indecomposable summand $U_i$ of $U$ exists in $\GP B$ and coincides 
with the mutation of $U_i$ as a summand of  $T$ in $\CM C$. This proves (3).

The modules $E_i$ in  \eqref{eq:mut}  and $F_i$ in \eqref{eq:mut2} determine 
the arrows incident to the vertices corresponding to $U_i$ in both $Q^\circ_U$ and  $Q_T^\circ$. 
So (4) follows.
 
Since $E_i$ and $F_i$ have no common summands and have no summands equal to $U_i$, 
because $Q_T$ has no loops and no 2-cycles (see \cite[Prop 4.2]{JKS2}), it follows
that $Q_U^\circ$ has no loops and no 2-cycles. 

Finally, $Q_{\mu_iT}^\circ=\mu_iQ_T^\circ$ (see \cite[Thm. 6.1]{JKS2}), together with
the properties of $Q_U$ just proved, implies that $Q_{\mu_iU}^\circ=\mu_iQ_U^\circ$.
Here, by slight abuse of notation, $\mu_iT$ denotes the mutation of $T$ at $U_i$.
This completes the proof.
\end{proof}

\begin{remark} 
The fact that $\GP B$ is a cluster subcategory of $\CM C$ that has no loops or 2-cycles 
also follows from \cite[Prop II.2.8]{BIRS}, Theorem \ref{thm:FF} and the fact that 
$\CM C$ has no loops and no 2-cycles (see \cite[Prop 4.2]{JKS2}). Above, we give a 
direct account to deduce the required finer relation (Theorem \ref{main1} (2) 
and (4)) between the cluster tilting objects in $\GP B$ and $\CM C$.
\end{remark}

We remark that $\catD_2$ is also a Frobenius 2-CY cluster category, however we 
do not pursue this further in this paper. We do however require the following
corollary.

\begin{corollary} \label{cor:mutinD2}
Let $T=U\oplus V$ be as in Theorem \ref{main1} (2) and let $T_i$ be a mutable 
summand which belongs $\add V$. If $E$ and $F$ denotes the middle 
terms in the mutation sequences  \eqref{eq:mutsequences} of $T_i$, 
then $E,F\in \add (B\oplus V)$.
\end{corollary}
\begin{proof}
This follows from Theorem \ref{main1} (4). 
\end{proof}

\subsection{Approximations in $\GP B$}

Let $U$ and $T=U\oplus V$ be cluster tilting objects in $\GP B$ and $\CM C$, respectively. We investigate the relations between $\add U$- and 
$\add T$-presentations, which will be used in the next section.  

Let $M, X\in \CM C$ (or in $\sub Q_k$). A short exact sequence  
\begin{equation}\label{eqaddpres1}
\xymatrix{0\ar[r] & X''\ar[r] &  X'\ar[r]^f &  M\ar[r] &0}
\end{equation}
is an {\it $\add X$-presentation} of $M$, if $X'', X'\in \add X$, and $f$ is an $\add X$-approximation.
Similarly, a short exact sequence 
\begin{equation}\label{eqaddpres}
\xymatrix{ 0\ar[r] & M\ar[r]^g &  X'\ar[r] &  X''\ar[r] & 0}
\end{equation}
is an {\it $\add X$-copresentation} of $M$, if $X'', X'\in \add X$, and $f$ is an 
$\add X$-approximation. In general, the map $f$ in an $\add X$-presentation is not
necessarily surjective. However, in our context, it will be. The map $g$ in 
an $\add X$-copresentation is however always injective for modules in $\CM C$. 

\begin{proposition}\label{mutgpBC}   
Let $M\in \CM B$ and $N\in \catD_2$.
\begin{enumerate}

\item  $M $ has an $\add U$-presentation.

\item $M$ has an  $\add U$-copresentation 
if and only if $M\in \GP B$.   

\item $N$ has an  $\add (B\oplus V)$-presentation
and an $\add (B\oplus V)$-copresentation.
\end{enumerate}
Moreover, these $\add U$- and $\add (B\oplus V)$-(co)presentations are also $\add T$-(co)presentations.
\end{proposition}
\begin{proof}
(1) is \cite[Lem. 3.3]{JKS3} and is included here for the completeness of the statements. 
 
(2) Let $M\in \GP B$. 
Let $f\colon M\to U'$ be an $\add U$-approximation. We have the following short exact sequence
\begin{equation}\label{eq:addUpres}
\xymatrix{
0\ar[r] &M \ar[r]^f & U' \ar[r] & U'' \ar[r] & 0.
}\end{equation}
As $\GP B$ is  a Frobenius category, $M$ has a cosyzygy $\Sigma M\in \GP B$. We have a short exact sequence, 
\begin{equation*}
\xymatrix{
0\ar[r] &M \ar[r]^g & B' \ar[r] & \Sigma M  \ar[r] & 0,
}\end{equation*}
where $B'\in \add B$. As $f$ is an $\add U$-approximation and $B'\in \add B\subseteq \add U$, 
$g$ factors through $f$. We have the following commutative diagram, 
\begin{equation*}
\xymatrix{
0\ar[r] &M \ar[r]^f \ar[d]^{\verteq} & U' \ar[r]\ar[d]^h & U'' \ar[r]\ar[d]^l & 0\\
0\ar[r] &M \ar[r]^g & B' \ar[r] & \Sigma M  \ar[r] & 0
}\end{equation*}
By adding additional summands from $\add U$ if needed, we may assume that $h$ is 
surjective and so $l$ is surjective as well. We have
$\ker l=\ker h\leq U' $ and $U'\in \CM B$. So $\ker l\in \CM B$. Now $U''$ is an extension 
of $\Sigma M\in \GP B\subset \CM B$ by $\ker l\in \CM B$. Therefore, $U'' \in \CM B$. Since 
$f$ is an $\add U$-approximation, 
\begin{equation}\label{eq:adhoc}
\Ext^1(U'', U)=0.\end{equation} 
Consequently, 
$\Ext^1(U'', B)=0$ and so $U''\in \GP B$, following the definition of $\GP B$ \eqref{eq:extCB}.
Furthermore, $U''\in \add U$, because
$U$ is a cluster tilting object in $\GP B$. 
 Therefore, \eqref{eq:addUpres} is an $\add U$-copresentation.

Conversely, suppose that $M\in \CM C$ has a copresentation as \eqref{eq:addUpres}. 
Then $M\in \CM B$ as it is the kernel of a homomorphism in $\CM B$. Applying $\Hom(B, -)$ to \eqref{eq:addUpres}, we have 
\[
\xymatrix{
0\ar[r] &\Hom(B, M) \ar[d]^{\verteq} \ar[r] & \Hom(B, U') \ar[d]^{\verteq} \ar[r] & \Hom(B, U'')\ar[d]^{\verteq} \ar[r] &\Ext^1_C(B, M) \ar[r]& 0\\
0\ar[r] &M \ar[r] & U' \ar[r] & U'' \ar[r] & 0.
}
\]
So $\Ext^1_C(B, M)=0$. Thus $\Ext^1(M, B)=0$ and so $M\in \GP B$.

(3) 
Recall that $N\in \catD_2$. 
We show that the minimal $\add T$-(co)presentation (which exists using
$B=C$ and (1) and (2)) is an $\add (B\oplus V)$-(co)presentation. First consider the minimal $\add T$-presentation, 
\begin{equation*}\label{eq:310}
\xymatrix{0\ar[r] & X''\ar[r] &  X'\ar[r]^f &  N\ar[r] &0,}
\end{equation*}
where the map $f$ is surjective because $C\in \add (B\oplus V)$ and $N\in \CM C$.
Any map $U\to N$ factors through $B\oplus C$
(see \eqref{eq:catD2}). So $X'$ has no direct summands in $\add U$. By Proposition \ref{prop:GPBV}, 
we have $\Ext^1_C(N, U)=0$. Therefore, $X''$  has no direct summands from 
$\add U$ either. 
So the  $\add T$-presentation is an $\add (B\oplus V)$-presentation. Similarly, the 
 minimal $\add T$-copresentation is an $\add (B\oplus V)$-copresentation. 

Finally,
it remains to show that the $\add U$- and $\add (B\oplus V)$-(co)presentations are $\add T$-(co)presentations. To achieve this, 
 applying $\Hom(T, -)$ to the presentations and $\Hom(-, T)$ to the 
copresentations gives short exact sequences, because $\Ext^1(T, T)=0$ and 
$T=U\oplus V$. Consequently, these (co)presentations are also $\add T$-(co)presentations.
\end{proof}

\subsection{Cluster subcategories of $\sub Q_k$}

In this subsection, we describe $\pi (\GP B)$ as a subcategory of $\sub Q_k$.
For $M\in \sub Q_k$, denote by $\M$ the minimal lift of $M$ to $\CM C$, namely, 
$\pi(\M)=M$ and $\M$ has no direct summand in $\add P_0$.
Let 
\begin{equation}\label{eq3}
\xymatrix{
0\ar[r]& \ker f\ar[r] & X\ar[r]^f & M \ar[r]&0,
}
\end{equation}
be a short exact sequence in $\sub Q_k$. The lift to a short exact sequence in $\CM C$
(which exists because $\pi$ induces an isomorphism on $\Ext^1$), 
gives an  exact sequence,  
\begin{equation}\label{eqliftof3}
\xymatrix{
0\ar[r]& \widehat{\ker f} \ar[r] & Y \ar[r] &\M \ar[r]&0,
}
\end{equation}
where $Y=\X\oplus (P_0)^r$, for some $r$.
\begin{lemma}\label{lem:liftses}
The module $Y$ in \eqref{eqliftof3} has no direct summand $P_0$ (i.e. $r=0$) 
if and only if the map $f$ in \eqref{eq3} is surjective on socles, that is, 
$f|_{\soc X}\colon \soc X\to \soc M$ is surjective.
\end{lemma}
\begin{proof}
Note  that $\rank Z=\dim \soc \pi Z$ for any $Z\in \CM C$ that has no direct summand 
$P_0$ (see \cite[Thm 4.5]{JKS1}), and rank is additive, in particular, 
\[
\rank Y=\rank \widehat{\ker f}+\rank \M.
\]
 Therefore, 
$Y$ has no direct summand $P_0$ if and only if 
\[\dim \soc X=\dim \soc \ker f+\dim \soc M,\] 
which is equivalent to $f$ being surjective on socles.
\end{proof}

\begin{definition}\label{de:gppib}
Let  $\GP \Bbar$ be the full subcategory of $\sub Q_k$
consisting of objects $M$ satisfying the following. 
\begin{enumerate}
\item[(1a)] if $P_0 \in \add B$, there exist $X\in \add \Bbar$ and a surjective map $f\colon X\to M$.

\item[(1b)] if $P_0\not\in \add B$, there exist $X\in \add \Bbar$ and a surjective map $f\colon X\to M$ 
such that $f$ induces a surjective map $\soc X\to \soc M$.

\item[(2)] $\Ext^1(M, \Bbar)=0$.
\end{enumerate}
\end{definition}

\begin{proposition}\label{prop:gpbbar} We have the following.
\begin{enumerate}
\item \label{itm:pbbbar1} $\pi (\GP B)=\GP \Bbar$. 
\item The category $\GP \Bbar$ is functorially finite in ${\sub Q_k}$, and
therefore a cluster subcategory of $\sub Q_k$. In particular, maximal rigid objects 
and cluster tilting objects in $\GP \Bbar$ coincide.
\end{enumerate}
\end{proposition} 
\begin{proof} (1) 
Note that the functor $\pi$ is full exact \cite{JKS1}. So for any $M\in \sub Q_k$,
\begin{equation}\label{eq:compext}
\Ext^1(\widehat{M}, B)=0 \text{ if and only if } \Ext^1(M, \pi B)=0.
\end{equation}
So, $\pi(\GP B)\subseteq \GP \Bbar$, regardless of whether Definition \ref{de:gppib}
(1a) or (1b) is satisfied, using Lemma \ref{lem:liftses} 
in the latter case. 

Let $M\in \GP \Bbar$. We can lift an $\add \pi B$-approximation $f\colon X\to M$ 
to a surjective map  
\[
\X\oplus (P_0)^r\to \M
\]
with $\M\in \CM C$. If $P_0\in \add B$, then $\M\in \CM B$ and $\M \in \GP B$ by \eqref{eq:compext}. If $P_0\not \in \add B$, then $r=0$, by Lemma \ref{lem:liftses}, 
and so again $\M\in \GP B$ by \eqref{eq:compext}.
Therefore $\pi(\GP B)=\GP \pi B$.

(2) By Theorem \ref{thm:FF} and (1), $\GP \Bbar$ is functorially finite  and extension closed. So (2) follows from \cite[Thm. II.3.1]{BIRS}.
\end{proof}

\begin{remark} Recall that $\Lambda$ is the preprojective algebra of type $\mathbb{A}_{n-1}$ and is 
artinian. The functorial finite subcategory  $\GP \Bbar$  is in general not as  described in \cite[Thm 5.10]{AS} (see also Remark \ref{rem:cmbfunct}).  Therefore, 
 $\GP \Bbar$ is different from those constructed in \cite{Lec} by Leclerc. However, 
 we will see that $\GP \Bbar$ still offers a categorification
of the positroid varieties described in \cite{Lec} (see Theorem \ref{thm:cat}).
\end{remark}

\section{Properties of  $\Psi$ and bases of $\CC[\Gr(k, n)]$} \label{sec:4}

We compare the cluster character $\Psi_M$ with Fu-Keller's \cite{FuKeller} 
cluster characters for any $M\in \GP B$. We then enhance the construction of bases 
of $\CC[\Gr(k, n)]$, whose elements have pairwise distinct minimal leading exponents \cite{JKS3}, to 
obtain bases,
whose elements additionally have pairwise distinct maximal leading exponents. 

\subsection{$g$-vectors and $cog$-vectors}

We require some notation for cluster categories $\catC$ more generally. Recall that our cluster
categories are always full exact subcategories of either $\sub Q_k$ or $\CM C$.
Let $T$ be a cluster tilting object in $\catC$ and $A=(\End T)\op$. Denote by 
$K(\CM A)$ (when $\catC\subseteq \CM C$), $K(\proj A)$ and $K(\fd A)$, the 
Grothendieck groups of Cohen-Macaulay $A$-modules, projective $A$-modules, and 
finite dimensional $A$-modules, respectively. 

Note that $A$ has finite global dimension (see, for instance, \cite[Prop 3.6]{JKS3} 
and \cite{GLS08}). When $\catC=\sub Q_k$, 
\[
K(\fd A)=K(\proj A),
\]
and when $\catC=\GP B$,
 
\begin{equation}\label{eq:GrothCMProj}
K(\CM A)=K(\proj A).
\end{equation}
For any $M\in \catC$, let 
\[\gv(M)=[\Hom(T, M)] \text{ and }
\cogv(M)=[\Hom(T, M)]- [\Ext^1(T, M)],
\]
which are the \emph{$g$-vector} and \emph{$cog$-vector} of $M$ in $K(\proj A)$, 
respectively. Note that when $\catC\subseteq \CM C$,  $\Ext^1(T, M)$
is not Cohen-Macaulay; however, we can view $[\Ext^1(T, M)]$ as a class in 
$K(\proj A)$ (or, equivalently, $K(\CM A)$), by taking projective resolutions (resp. 
CM-approximations). We denote $[T,M]=[\Hom(T,M)]$, for simplicity.

\subsection{Comparison of cluster characters}

Let $T$ be a cluster tilting object in $\CM C$ and $A=A_T=(\End T)\op$.
Let 
\[
\beta=\beta_T\colon K(\fd A)\to K(\CM A)
\]
be the map induced by taking projective resolutions or equivalently by taking 
$\CM$-approximations (see \eqref{eq:GrothCMProj}). 
On mutable simples $S_i$, we have 
\begin{equation}\label{eq:betaint}
\beta_T([S_i])=[T,E_i]-[T,F_i], 
\end{equation}
which can be  obtained by applying $\Hom(T, -)$ to 
the mutation sequences of $T_i$ given in \eqref{eq:mutsequences}.

Fu-Keller's cluster character can be adapted to $\CM C$ as follows~(see \cite{JKS3}),  
\begin{equation} \label{eq:FKchar}
\begin{aligned}
\Phi^T_M & =x^{[T, M]} \sum_{d} \chi(\Gr_d\Ext^1(T, M)) x^{-\beta(d)} \in \CC[K(\proj A)]\\
& =x^{\cogv(M)}\sum_{d} \chi(\Gr^d\Ext^1(T, M)) x^{\beta(d)} \in \CC[K(\proj A)],
\end{aligned}
\end{equation}
where $\Gr_d \Ext^1(T,M)$ and $\Gr^d \Ext^1(T,M)$  are, respectively, the Grassmannians of submodules and quotient modules of $\Ext^1(T,M)$ 
with dimension vector $d$, and $\chi$ is the Euler characteristic. 

Let $U\in \GP B$ and  $T=U\oplus V$ be cluster tilting objects as described in 
Theorem~\ref{main1}, and let $A_U=(\End U)^{\mathrm{op}}$. Define 
\[
\gamma\colon K(\proj A_U) \to K(\proj A_T)\colon [U, U_i]\mapsto [T, U_i].
\]
Then $\gamma$ induces a homomorphism of algebras, 
\[
\gamma^*\colon \CC[K(\proj A_U)] \to \CC[K(\proj A_T)].
\]
\begin{proposition} \label{prop:charongpb}
For any  $M\in \GP B$, $\gamma^*(\Phi^U_M)= \Phi^T_M$.
\end{proposition}
\begin{proof} Let $M\in \GP B$ and let
\begin{equation*}
\xymatrix{0\ar[r] & U''\ar[r] &  U'\ar[r] &  M \ar[r] &0}
\end{equation*}
be an $\add U$-presentation of $M$. We have $[U, M]=[U, U']-[U, U'']$. 
Applying $\Hom(T, -)$ to the $\add U$-presentation, we have 
\[
 [T, M]=[T, U']-[T, U''] \text{ and so } \gamma([U, M])=[T, M].
\]

Now suppose that $U_i$ is a mutable summand in $\GP B$ with  
the mutation sequences, 
\begin{equation*}
\xymatrix{0\ar[r] & U_i\ar[r] &  F_i\ar[r] &  U^*_i \ar[r] &0, &
 0\ar[r] & U^*_i\ar[r] &  E_i\ar[r] &  U \ar[r] &0.
}
\end{equation*}
Then by \eqref{eq:betaint},
\begin{equation}\label{eq:betaUT}
\beta_U([S_i])=[U, F_i]-[U, E_i]  \text{ and so } \gamma(\beta_U([S_i])) =\beta_T([S_i]),
\end{equation}
because mutation sequences in $\GP B$ are also mutation sequences in $\CM C$, by 
Theorem~\ref{main1} (3). 
Note that for any $M\in \GP B$, by Proposition \ref{prop:GPBV}, 
$\Ext^1(T, M)=\Ext^1(U, M)$ and so the corresponding Grassmannians of submodules 
are the same. In particular,
\[
\chi(\Gr^d\Ext^1(T, M))=\chi(\Gr^d\Ext^1(U, M)).
\]
 for all $d$.  Therefore, $\gamma^*(\Phi^U_M)=\Phi^T_M$.
\end{proof}

\begin{corollary}\label{rem:char} 
Let $U\in \GP B$ and $T=U\oplus V\in \CM C$ be cluster tilting objects. 
For any indecomposable summand $T_i\in \add T$ and $U_i\in \add U$,
by substituting $\Phi^T_{T_i}=\Psi_{T_i}$,  
and $\Phi^U_{U_i}=\Phi^T_{U_i}$, we have 
$\Psi_M=\Phi^T_M=\Phi^U_M$ as functions on $\CC[\Gr(k, n)]$ for any $M\in \GP B$.
\end{corollary}
\begin{proof}
By \cite[Thm 9.11]{JKS3}, we have $\Psi_M=\Phi^T_M$, which is $\Phi^U_M$, 
by Proposition \ref{prop:charongpb}.
\end{proof}

\subsection{Bases of $\CC[\Gr(k, n)]$}

To help the reader understand the construction in Section \ref{Sec:cog}
below, we recall the construction of such bases from \cite{JKS3} in some
detail. This construction builds on work of Lusztig \cite{Lus00} and Geiss-Leclerc-Schr\"{o}er \cite{GLS05} for preprojective algebras 
(see Plamondon \cite{PL} for a related construction).

 Let $\rep(\preproj, d)$ be the representation variety of 
$\preproj$ with dimension vector $d\in \NN^{n-1}$.
Note that $\GL{d}=\prod_{i=1}^{n-1} \GL{d_i}$ acts on $\rep(\preproj, d)$ by conjugation  
such that $\GL{d}$-orbits are in one-to-one correspondence with the isomorphism 
classes of representations of $\preproj$ of dimension vector $d$.

For any $d\in \NN^{n-1}$ 
and any $\preproj$-module $X$, the map
\begin{equation}\label{eq:usc-dim-hom}
  \rep(\preproj, d) \to \NN \colon M \mapsto \dim \Hom(X, M)
\end{equation}
is upper semicontinuous, 
that is, $\{ M\st \dim \Hom(X, M) \leq m\}$ is open, for any $m\in\NN$.
Let
\[
  \sub(\preprojQ_k, d)\subset \rep(\preproj, d),
\] 
be the $\GL{d}$-invariant subset of $\preproj$-modules in $\sub\preprojQ_k$, and let
\[
  \sub(\preprojQ_k, d,r)\subset \sub(\preprojQ_k, d),
\] 
be the $\GL{d}$-invariant subset of modules $M$ with $\dim\soc M\leq r$.

\begin{lemma}\label{Lem:quotrep}\cite[Lem 10.3]{JKS3}
Both $\sub(\preprojQ_k, d)$ and $\sub(\preprojQ_k, d,r)$ are open in $\rep(\preproj, d)$. 
\end{lemma}

We will work with cluster tilting objects in $\sub Q_k$ and $\CM C$. 
To avoid confusion, we use overline for modules in $\sub Q_k$.
Let $\olT=\bigoplus_i \olT_i$ be a cluster tilting object in $\sub\preprojQ_k$, 
$\olA =(\End \olT)\op$ and $\olQ$ the Gabriel quiver of $\olA$.
For any $M\in \sub\preprojQ_k$, denote by $[\olT, M]$ the class of the $\olA$-module $\Hom( \olT, M)$ 
in the Grothendieck group $\Grot(\fd \olA)$ of finite dimensional $\olA$-modules. 
By identifying the basis of simple classes with the standard bases, 
we have 
\begin{equation}\label{eq:grothZlat}
\Grot(\fd \olA)=\ZZ^{\olQ_0},
\end{equation} and thus 
$ [\olT, M]=\dimv\Hom(\olT, M).$ Also, we have $\Grot(\proj \olA)=\Grot(\fd \olA)$ since $\olA$ has finite global dimension. 

Denote by $\IrrSub(d)$ the set of irreducible components in $\sub(\preprojQ_k, d)$ and let 
\[
  \IrrSub=\bigcup_d \IrrSub(d).
\] 
By the semicontinuity of \eqref{eq:usc-dim-hom},
for each $\hC\in \IrrSub$, there is a class 
\begin{equation}\label{eq:gg} \gminvec{\hC}\in\GrotAbar, \end{equation}
which is minimal (as a dimension vector) in $\{[\olT, N]\st N\in \hC\}$. 
Furthermore, the set  
\[
  \hCmin=\{N\in \hC\st [\olT, N]=\gminvec{\hC}\}
\] 
is an open dense subset of $\hC$. We call $\gminvec{\hC}$
the \emph{generic $g$-vector} of~$\hC$. This defines the map
\[
\gminmap\colon\IrrSub \to \GrotAbar\colon \hC\mapsto \gminvec{\hC}.
\]
\begin{proposition} \cite[Cor. 10.9, Prop. 10.13]{JKS3} \label{prop:generic-gv}
The map $\gminmap$ is injective. Moreover, any $g$-vector is generic.
\end{proposition}

The result can be lifted to $\CM C$ to obtain a basis, whose elements have pairwise 
distinct minimal leading exponents
for $\CC[\Gr(k, n)]$.  Let us recall some relevant notation. 
First, the Grothendieck groups $K(\CM C)$ and $K(\sub Q_k)$ satisfy the following relation, 
\begin{equation}\label{eq:Grot-isom}
 \Grot(\CM C)\cong \ZZ\times \Grot(\sub\preprojQ_k) \text{ with } [M]\mapsto (\rk [M],[\olM]),
\end{equation}
where $\rk\colon K(\CM C) \to \ZZ, ~ [M]\mapsto \rank_{C} M$. Denote the isomorphism by $\theta$.

For $\lambda\in\Grot(\CM C)$, 
let $\CM(\lambda)$ be the set of isomorphism classes in $\CM C$ of modules of class $\lambda$
and let
\[
  \CM(r)=\bigcup_{\rk\lambda=r} \CM(\lambda).
\]
If $\theta(\lambda)=(r,d)$, 
then $\pi$ determines a bijection between $\CM(\lambda)$ and the set of $\GL{d}$-orbits in 
$\sub(\preprojQ_k,d,r)$,
which is an open subset of $\sub(\preprojQ_k,d)$ by Lemma \ref{Lem:quotrep}. 
When $\hC\in\IrrSub(d)$ intersects $\sub(\preprojQ_k,d,r)$, define
\begin{equation}\label{eq:Crc}
  \HC(r, \hC)=\{ M\in \CM(\lambda)\st \pi M \in \hC \}. 
\end{equation}
By abuse of terminology, $\HC(r, \hC)$ is called an irreducible component of 
$\CM(\lambda)$. Let $\IrrCM(\lambda)$ be the collection of all such components,  
and let 
\[ 
 \IrrCM(r) = \bigcup_{ \rk\lambda=r} \IrrCM(\lambda),\quad 
 \IrrCM = \bigcup_r \IrrCM(r).
\]

For any $\HC\in \IrrCM(r)$, define
\begin{equation}\label{eq:Crc-too}
  \HC_{\gen}=\{ M\in \HC \st \pi M \in \hC_{\gen} \}. 
\end{equation}
Note that the g-vectors of modules in $\HC_{\gen}$ are constant and
 are referred to as the \emph{generic g-vectors} of $\HC$.
Choose $N_\HC\in \HC_\gen$ for each $\HC\in\IrrCM(r)$.

Let  $\mathsf{J}=\Hom_\ring(e_0C, \ring)$, 
where $e_0$ the idempotent corresponding to vertex 0 in the circular double 
quiver defining the algebra $C$.  Recall the weight map $\wt$ from \cite{JKS3}, 
\[
\wt\colon K(\CM A)\to K(\fd A), ~
[X]\mapsto (\rk X)[T, \mathsf{J}]-[X]. 
\]
As in \eqref{eq:grothZlat}, we can identify $K(\fd A)$ with $\ZZ^{Q_0}$. 
The usual partial order on $\ZZ^{Q_0}$ induces a partial order on $K(\CM A)$ 
(\cite[\S 5]{JKS3}), that is, for $\lambda, \lambda'\in K(\CM A)$, we say 
\begin{equation} \label{eq:partord}
\lambda\leq \lambda' \text{ if }
\rk\lambda <\rk\lambda', \text{ or }\rk\lambda =\rk\lambda' \text{ and } \wt\lambda \leq \wt\lambda'.
\end{equation} 

\begin{proposition}\label{Thm:GrassGeneric} \cite[Cor. 10.23]{JKS3}
The functions in $\{\Psi_{N_\HC}\colon \HC\in \IrrCM\}$ have distinct (minimal) 
leading exponents (i.e. the generic $g$-vectors), and form a basis of $\CC[\Gr(k, n)]$. 
Consequently, any $g$-vector is generic. 
\end{proposition}

\subsection{Generic cog-vectors} \label{Sec:cog}

Similar to \eqref{eq:gg},  by the semicontinuity of 
\[
\Hom_\Lambda(-, \olT_j)\colon \sub(Q_k, d)\to \NN,
\]  
for each component $\hC\in \IrrSub(d)$, there is a class 
\begin{equation}\label{eq:ggop} 
\ggop(\hC)\in K(\fd \olA^{\op}), 
\end{equation}
which is minimal (as a dimension vector) in $\{[N, \olT]\colon N\in \hC\}$. 
Moreover, analogous to Proposition~\ref{prop:generic-gv}, we have an injective map, 
\begin{equation}\label{eq:ggopinj}
\ggop\colon\IrrSub \to \GrotAbar\colon \hC\mapsto \ggop(\hC).
\end{equation}

Next, recall the bilinear form defined on $K(\fd \Lambda)$,  
\begin{equation}\label{Eq:bilin2}
\langle  c, d\rangle=2\sum_{i=1}^{n-1} c_id_i- \sum_{i=1}^{n-2} (c_id_{i+1}+c_{i+1}d_i),
\end{equation}
where $d=(d_i)$ and $c=(c_i)$.
Crawley-Boevey \cite{CB} gave a homological interpretation  of the bilinear form as follows, 
\begin{equation}\label{eq:CB}
\langle [X], [Y]\rangle=
\dim \Hom_\Lambda( X,Y) + \dim \Hom_\Lambda( Y, X) 
- \dim \Ext_\Lambda^1(X,  Y),
\end{equation}
where $X, Y\in \mod \Lambda$. In particular, for any summand $T_j\in \add T$, 
\begin{equation}\label{eq:CB1}
\begin{aligned}
\langle [\olT_j], [Y]\rangle&=
\dim \Hom_\Lambda( \olT_j,Y) + \dim \Hom_\Lambda( Y, \olT_j) 
- \dim \Ext_\Lambda^1(\olT_j,  Y) \\
&= \cogv(Y)_j+ \dim  \Hom_{\Lambda}(Y,  \olT_j). 
\end{aligned}
\end{equation}
Identify both $K(\fd \olA^{\op})$ and $K(\fd \olA)$ with $\ZZ^{\olQ_0}$ (see \eqref{eq:grothZlat}). 

Using that $\langle [\olT_j], -\rangle$ is constant on $\rep(\Lambda, d)$,
we see that \[\cogv(Y)_j=\dim \Hom_\Lambda( \olT_j,Y) - \dim \Ext_\Lambda^1(\olT_j,  Y)\]
is lower semicontinuous (which is also known from \cite{GSXX}).
By \eqref{eq:ggop} and \eqref{eq:CB1}, there is a class
\[ \gcog(\hC)\in K(\fd  \olA), \] 
which is maximal in $\{\gcog(N)\st N\in \hC\}$. 

Furthermore, the set  
\[
\hCminco=\{N\in \hC\st \cogv(N)=\gcog(\hC)\}
\] 
is an open dense subset of $\hC$. We call $\cogv_{\gen}(\hC)$ the 
\emph{generic cog-vector} of $\hC$. Define 
\[
\gcog\colon \IrrSub \to \GrotAbar\colon \hC\mapsto \gcog(\hC).
\]
Now, by the injectivity of $\ggop$ in \eqref{eq:ggopinj} and \eqref{eq:CB}, we have 
the following.

\begin{proposition}\label{prop:gen-gv-cogv}
The map $\gcog$ is injective.
Moreover, any $cog$-vector is generic.
\end{proposition}

For any $\HC\in \IrrCM$ with $\pi \HC=\hC\in \IrrSub$, choose $N_\HC$ such that 
\[
\pi N_\HC\in \hCmin\cap \hCminco.
\]
Lifting Proposition \ref{prop:gen-gv-cogv} to $\CM C$, we obtain a finer version of 
Proposition \ref{Thm:GrassGeneric}. 

\begin{proposition}\label{Thm:GrassGen-gv-cogv} 
The functions in the basis $\{\Psi_{N_\HC}\colon \HC\in \IrrCM\}$ have pairwise distinct 
minimal and distinct maximal exponents. The minimal and maximal exponent of  
$\Psi_{N_\HC}$ is the generic $g$-vector and the generic $cog$-vector of 
$\HC$, respectively.  
\end{proposition}

We do not claim the basis in Proposition \ref{Thm:GrassGen-gv-cogv} 
is unique. However, it is known \cite{Lus00} \cite{GLS05} 
that there is a non-empty open subset of each irreducible component 
where the cluster character $\psi$ defined on $\sub Q_k$ is constant. 
The corresponding basis is the dual semicanonical basis
for the unipotent cell. The lift of this basis to $\CC[\Gr(k, n)]$ is of the form 
in Proposition \ref{Thm:GrassGen-gv-cogv}, and 
will be called \emph{the generic basis} of $\CC[\Gr(k, n)]$ in 
the sequel.

\subsection{The cluster character $\Psi_{M\oplus B^r}$}
Let $B$ be an algebra as defined in Section \ref{sec:GPB}, and let $\catD$ 
be as defined in \eqref{eq:catD}. In this subsection, we will see that for 
any $M\in \CM C$ and $r\gg 0$, the cluster character $\Psi_{M\oplus B^r}$ is a 
linear span of $\Psi_X$ with $X\in \catD$. 

\begin{proposition}\label{prop:indstep}
Let $M\in \CM B\backslash \GP B$.
\begin{enumerate}
\item $\Ext^1(M, B)\not= 0$. 

\item For $B'\in \add B$ such that $\Ext^1(M, B')\not= 0$, and any non-split extension,
\begin{equation}\label{eq:ses}
\xymatrix{
0\ar[r] & B'\ar[r]^b & X \ar[r]^c &M\ar[r] &0,
}
\end{equation}
we have
\[
\dim \Ext^1(X, B') < \dim \Ext^1(M, B').
\]
\item The map $c$ in \eqref{eq:ses} is a $\GP B$-approximation if and only if 
$X\in \GP B$.
\end{enumerate}
\end{proposition}
\begin{proof} (1) follows from the definition of $\GP B$ \eqref{eq:defpgb}. Applying $\Hom(-, B')$ to the 
non-split short exact sequence \eqref{eq:ses}, 
we have a long exact sequence as follows,
\[
\xymatrix{
\Hom(X, B') \ar[r] &\Hom(B', B')\ar[r]^{\alpha} &
\Ext^1(M, B')\ar[r] & \Ext^1(X, B')\ar[r]&0.
}
\]
In particular, the connection map $\alpha$ is not zero, because the extension is not split. 
Consequently, we have the required inequality.

Finally, note that for any $Z\in \GP B$, we have $\Ext^1(Z, B)=0$. So applying 
$\Hom(Z, -)$ to \eqref{eq:ses} gives a short exact sequence. So  (3) follows.  
\end{proof}

By abuse of notation,  we write $X\in \Ext^1(M, N)$ to mean that  $X$ is 
the middle term of a short exact sequence, 
\[
\xymatrix{
0\ar[r] & N \ar[r] & X\ar[r] & M\ar[r] & 0.
}
\]

Let $M, N\in \CM C$ with $\Ext^1(M, N)\not=0$. For  $X\in \Ext^1(M, N)$,
let 
\[
\Ext^1(M, N)_{X}=\{Y\in \Ext^1(M, N)| \Psi_Y=\Psi_X\}, 
\]
which is a constructible subset of $\Ext^1(M,N)$ (see \cite[\S 1.2]{GLS07}. 
By homogenising the multiplication formula in \cite[Thm 1]{GLS07}, we have
\begin{equation}\label{eq:liftmult}
 \Psi_{M\oplus N}=
\sum_{X\not\cong M\oplus N}\frac{\chi(\mathbb{P} \Ext^1(M, N)_X)+\chi(\mathbb{P} \Ext^1(N, M)_X)}{\chi(\mathbb{P} \Ext^1(M, N))}\Psi_X.
\end{equation}

\begin{proposition}\label{prop:mainproperty}
Let $M\in \CM C$. Then there exists some $r\geq1$  
such that 
\[
\Psi_M\Psi_{B^r}=\sum_{X\in\catD} a_X\Psi_X,
\]
where the coefficients $a_X\in \QQ$ and only finitely many of them are not zero.
\end{proposition}
\begin{proof} If $M\in \catD$, then the result is clearly true. 
Next, assume that $M\not\in \catD$. Then  $\Ext^1(M, B)\not=0$. 
Following  \eqref{eq:liftmult}, we have 
\begin{equation}\label{eq:char}
\Psi_M\Psi_{B}=\sum_{X\not\cong M\oplus B} a_X\Psi_X,
\end{equation}
where the coefficients $a_X\in \QQ$ and only finitely many of them are not zero. Those $X$ in the sum 
with $a_X\not=0$ is an extension of $X$ and $B$.
By Proposition \ref{prop:indstep} (2) and its dual version, for those $X$  in the sum with 
$a_X\not=0$, 
we have 
\[
\dim \Ext^1(X, B) < \dim \Ext^1(M, B).
\]
Now the result follows by induction. 
\end{proof}

\section{More on $(co)g$-vectors}

We prove some more properties of $(co)g$-vectors. Let  $U$ be a cluster tilting
object in $\GP B$ and $T=U\oplus V$ be a cluster tilting object in $\CM C$.
\begin{lemma}\label{lem:cog}
Let $d=[T, T^+]-[T, T^-]\in K(\CM A)$, where $T^+, T^-\in \add T$. Then $d$ is a 
$cog$-vector if and only if there is a surjective map $T^+\to T^-$.
\end{lemma}
\begin{proof}
If there is a surjective map $f\colon T^+\to T^-$, then $d=\cogv(\ker f)$ 
is a $cog$-vector. We prove the converse. Suppose that $d=[T,T^+]-[T,T^-]=\cogv(N)$ 
and $N$ has the following $\add T$-copresentation, 
\begin{equation*}
\xymatrix{0\ar[r] & N\ar[r] &  T'\ar[r]^g &  T'' \ar[r] &0}.
\end{equation*}
We can assume that $N$ is such that $T'$ and $T''$ have no common summands.
Indeed, let X be the largest common summand and decompose $g$ as 
$\begin{pmatrix}g_1 & g_2\\g_3 & g_X\end{pmatrix}$, where $g_X\in \End X$. 
Now choose $a\in \CC$ so that $a\mathrm{Id}_X+g_X$ is an isomorphism and   
$g'=\begin{pmatrix}g_1 & g_2\\g_3 & a\mathrm{Id}_X+g_X\end{pmatrix}\colon T'\to T''$ 
is still surjective, where $\mathrm{Id}_X$ is the identity map on $X$. Then $X$ 
can be split off to obtain a new copresentation as claimed. (The existence of $a$ 
can be shown using linear algebra over $\ring$, or by lifting the corresponding 
result from $\sub Q_k$.) After removing common summands from $T^+$ and 
$T^-$ we obtain $T'$ and $T''$, respectively. Since there is a surjection $T'\rightarrow T''$, there is also a surjection $T^+\rightarrow T^-$, as needed.
\end{proof}

Let $\oplus_i T_i$ be the basic cluster tilting object associated to $T$. The   
exchange matrix $B_T=(b_{ij})$ corresponding to $T$ is defined by 
\begin{equation}
\label{eq:BfromT}
  b_{ij}=\dim \Irr(T_j, T_i)-\dim \Irr(T_i, T_j),
\end{equation}
where $T_j$ is mutable, $T_i$ can be any summand, and
$\Irr(T_j, T_i)$ and  $\Irr(T_i, T_j)$ are the spaces of irreducible maps between 
 $T_j$ to $T_i$ in $\add T$.

Denote by $K(\fd A)_{\mathrm{int}}$ the subgroup of $K(\fd A)$, generated by 
the classes of the simples corresponding to mutable vertices.  
We say that $d\in K(\fd A)$  \emph{is internally supported}
if $d\in K(\fd A)_{\mathrm{int}}$.

\begin{lemma}\label{lem:betainj}
Let $U$ be  a cluster tilting object in $\GP B$ and $A_U=(\End U)\op$.
The restriction $\beta_U: K(\fd A_U)_{\mathrm{int}}\to K(\CM A_U)$ is injective.
Consequently, the exchange matrix $B_U$ has full rank.
\end{lemma}
\begin{proof} 
Let $T=U\oplus V$ be a cluster tilting object in $\CM C$ with $V\in \catD_2$.
Then $\beta_T$ is injective when restricted to $K(\fd A)_{int}$ using  \cite[Prop 5.5]{JKS3} 
and therefore so is $\beta_U$ using that $\gamma(\beta_U([S_i]))=\beta_T([S_i])$ 
(see \eqref{eq:betaUT}), for a simple $A_U$-module (resp. simple $A_T$-module) 
$S_i$.
\end{proof}

\begin{remark}
The full rank of $B_U$ was shown in \cite[Prop 4.9]{GL} using a different 
approach. This property implies that the cluster algebra defined by $B_U$ 
(see \S \ref{SEC:GPpositrod}) can be quantised.
\end{remark}

Let $T(d)=\oplus T_i^{d_i}$, and let $E(d)$ and $F(d)$ be defined similarly.  
We have a long exact sequence 
\[
\xymatrix{
0\ar[r] & T(d) \ar[r] & F(d) \ar[r] & E(d)\ar[r] & T(d)\ar[r] & 0,
}
\]
and so
\begin{equation}\label{eq:d}
\beta(d)=[T, F(d)]-[T, E(d)].
\end{equation}

Let $M\in \CM C$ with a copresentation
\begin{equation}\label{eq:Mcopres}
\xymatrix{0\ar[r]& M\ar[r] &  T'\ar[r] &  T'' \ar[r] & 0.}
\end{equation}

The following result gives a criterion for when an arbitrary term in the cluster
character $\Phi_M^T$ (see  \eqref{eq:FKchar}) has an exponent which
is a $cog$-vector. 

\begin{proposition}\label{prop:cog}
The vector $\cogv(M)+\beta(d)\in K(\CM A_T)$ is a $cog$-vector if and only if there is a surjective map 
$T' \oplus F(d)\to T'' \oplus E(d)$.
\end{proposition}
\begin{proof}
By \eqref{eq:d}, we have 
\begin{equation}\label{eq:cogMd}
\begin{aligned}
\cogv(M)+\beta(d) & =[T, T']-[T, T'']+ [T, F(d)]-[T, E(d)] \\
&= [T, T'\oplus F(d)]-[T, T''\oplus E(d)].
\end{aligned}
\end{equation}
Now the proposition follows from Lemma \ref{lem:cog}.
\end{proof}

Proposition \ref{prop:cog} has the following surprising consequence. 

\begin{corollary}\label{cor:cog}
For any $M\in \CM C$, there exists a projective module $P$ such that the exponents  
of the terms with nonzero coefficient  in the sum \eqref{eq:FKchar} for 
$\Phi^T_{M\oplus P}$ are all $cog$-vectors.
\end{corollary}
\begin{proof}
By adding a projective module $P$ to the first two terms of the copresentation  \eqref{eq:Mcopres} of $M$, we obtain a copresentation of $M\oplus P$, 
\[
\xymatrix{0\ar[r]& M\oplus P \ar[r] &  T' \oplus P\ar[r] &  T'' \ar[r] & 0.}
\]
Therefore, 
\[
\cogv(M\oplus P)+\beta(d)=[T, T'\oplus P\oplus F(d)]-[T, T'' \oplus E(d)].
\]
By \eqref{eq:Mcopres}, there is a surjective map $T'\to T''$ and we can choose $P$ 
large enough so that there is a surjective map $F(d)\oplus P\to E(d)$ for those $d$ 
in \eqref{eq:FKchar} with nonzero coefficients. So we have a surjective map 
$T'\oplus P\oplus F(d)\to T''\oplus E(d)$. Now the corollary follows from 
Proposition \ref{prop:cog}.
\end{proof}

We refine the partial order \eqref{eq:partord} on $K(\CM A)$ to a total order by fixing a lexicographic order on the vertices of $Q$.

Let $M\in \CM C$. Write 
\begin{equation}\label{eq:lincomb}
\Psi_M=\sum_{i=1}^s a_{i} \Psi_{N_i}
\end{equation} 
as a linear combination of the basis in Proposition \ref{Thm:GrassGen-gv-cogv}, where 
$a_i\not=0$. 
We may identify  $\Psi_M$ with $\Phi^T_M$ \cite[Thm. 9.11]{JKS3} (see also Corollary \ref{rem:char}). 
Applying  \eqref{eq:FKchar}, 
\begin{equation}\label{eq:expansions}
\sum_{d} \chi(\Gr_d\Ext^1(T, M)) x^{[T, M]-\beta(d)} =
\sum_{i=1}^s a_i\sum_{d} \chi(\Gr_d\Ext^1(T, N_i)) x^{[T, N_i]-\beta(d)}.
\end{equation}

For $d\in K(\fd A)$, let $\supp(d)=\{i | d_i\neq 0\}$ and $\supp(M)=\supp([M])$
for a finite dimensional $A$-module $M$.

\begin{proposition} \label{prop:CompExt}
For any $N_i$ in \eqref{eq:lincomb}, we have  
\begin{equation}\label{eq:claimgv}
\gv(N_i)=\gv(M)-\beta(d) \text{ with } \supp(d)\subseteq \supp(\Ext^1(T, M)),
\end{equation}
and $\dimv \Ext^1(T, N_{i})\leq \dimv\Ext^1(T, M)$.
\end{proposition}
\begin{proof} 
By definition the exponents on the left hand side of \eqref{eq:expansions} are all of 
the form
\begin{equation}\label{eq:claim}
[T, M] - \beta(d) \text{ for some } 0\leq d\leq \dimv\Ext^1(T, M).
\end{equation}
In particular, such a vector $d$ is internally supported. The minimal exponents of 
$\Phi^T_{N_i}$ are distinct (see Proposition \ref{Thm:GrassGeneric}). 
Without loss of generality, we may assume that $[T, N_i]<[T, N_{i+1}]$ for all $i$. 
Comparing minimal exponents in $\Phi^T_M$ and $\Phi^T_{N_1}$ yields that 
\begin{equation}\label{eq;basecase}
[T, M]=[T, N_1] \text{ and } a_1=1.
\end{equation} 
If $s=1$,  then $\Psi_M=\Psi_{N_1}$ and so $\Phi^T_{N_1}$ clearly satisfies 
\eqref{eq:claimgv}. Now let $s>1$. We have 
\begin{equation}\label{eq:ww2}
\Phi^T_M-\Phi^T_{N_1}=\sum_{i=2}^s a_{i} \Phi^T_{N_i}.
\end{equation}
Since $[T, M]=[T, N_1]$ (i.e., $\gv(M)=\gv(N_1)$), the exponents of the terms in 
$\Phi^T_{N_1}$ and therefore all exponents on the left hand side are of the form 
\begin{equation}\label{eq:Newtry}
\gv(M)-\beta(d)  \text{ for some internally supported } d\geq 0, 
\end{equation}
Now repeatedly comparing minimal exponents to \eqref{eq:ww2}, we can conclude 
that all the exponents in the sum of $\Phi^T_{N_i}$ are of the form \eqref{eq:Newtry}.
Moreover, similar arguments using Propositon \ref{Thm:GrassGen-gv-cogv} imply  
that the exponents in $\Phi^T_{N_i}$, for all $i$, have the form 
\begin{equation}\label{eq:claimcog}
\cogv(M) + \beta(e) \text{ for some internally supported vector } e\geq 0.
\end{equation}
In particular, 
\[
[T, N_i]- \beta(\dimv \Ext^1(T, N_i)) =[T, M]- \beta(\dimv \Ext^1(T, M)) +\beta(e)
\]
for some internally supported  $e\geq 0$. Therefore,
\[
\beta(\dimv\Ext^1(T, M)-\dimv\Ext^1(T, N_i))= [T, M]-[T, N_i]+\beta(e)=\beta(d+e), 
\]
for some internally supported $d\geq 0$.

Recall that $\beta$ is injective, when it is restricted to the interior part of $K(\fd A)$ 
(see Lemma \ref{lem:betainj}). 
Therefore, 
\[
\dimv\Ext^1(T, M)-\dimv\Ext^1(T, N_i)=d+e\geq 0.
\]
Consequently,  \begin{equation}\label{eq:try2}\dimv \Ext^1(T, N_{i})\leq \dimv\Ext^1(T, M),\end{equation}
proving the second statement in the proposition.
Next, we prove the first statement. 
When $s=1$, it follows from \eqref{eq;basecase}. Now let $s>1$.
Comparing the minimal exponents of both sides in \eqref{eq:ww2}, we have 
\[
\gv(N_2)=\gv(M)-\beta(d)
\]
for some $d\leq \dimv \Ext^1(T, M)$ or $d\leq \dimv \Ext^1(T, N_1)$. In either case, 
by \eqref{eq:try2}, we have $d\leq \dimv \Ext^1(T, M)$.
Now applying comparison of the minimal exponents to 
\[
\Phi^T_M-\Phi^T_{N_1}-a_2\Phi^T_{N_2}=\sum_{i=3}^s a_{i} \Phi^T_{N_i},
\]
and repeating the procedure, we can conclude that for any $i$, 
\[
\gv(N_i)=\gv(M)-\beta(d) \text{ with } \supp(d)\subseteq \supp(\Ext^1(T, M)). 
\]
This completes the proof.
\end{proof}

\section{Plabic graphs, Gabriel quivers and associated cluster algebras}
Recall that the Grassmann necklace corresponding to $B$ (see  \eqref{eq:BN})  is 
\begin{equation}\label{eq:neckl}
\nlace=\{I\colon M_I\in\add B\}.
\end{equation} 
By Corollary \ref{cor:rk1oid}, the associated positroid $\positr=\{I\colon M_I \in \CM B\}$. So
\begin{equation}\label{eq:opoidring}
\CC[\poid]=\CC[\Gr(k, n)]/\langle \Delta_I\colon M_I\not\in \CM B\rangle, 
\end{equation} 
and $\CC[\opoid]=\CC[\poid]_B$, the localisation of $\CC[\poid]$ at
 $\{\Delta_I\colon M_I\in \add B\}$.

\subsection{A conjecture by Muller-Speyer}
Muller--Speyer constructed an (ice) quiver $Q_D$ from a Postnikov diagram $D$ \cite{MS} as follows. 
The vertices of $Q_D$ are the alternating regions of $D$. %
There is an edge from vertex $R_i$ (i.e. region $R_i)$
to vertex $R_j$ (i.e. region $R_j$) if the two regions are separated by a pair of strands which cross 
the borders of $R_i$ and $R_j$. 
The faces of $Q_D$ are in bijection with the clockwise and anticlockwise regions of $D$. 
The quiver $Q_D$ is oriented so that orientations of faces of $Q_D$ are consistent with the orientations of the corresponding regions of $D$. Note that $Q_D$ 
can also be constructed from the plabic graph $G$ corresponding to $D$, and we also denote the quiver 
by $Q_{G}$. Muller--Speyer did not consider arrows between the boundary 
regions, that is, $Q_D$ does not include arrows between boundary vertices (or frozen vertices).

Let $\algMS$ be the cluster algebra defined by $Q_D$ localised at the the frozen cluster variables, i.e. those 
cluster variables corresponding to the boundary vertices.
Muller-Speyer proposed the following. 
\begin{conjecture}\label{conj:MS}
\cite[Conj. 3.4]{MS}  $\algMS\cong \CC[\opoid]$.
\end{conjecture}

\subsection{Mutation of cluster tilting objects $M_\maxNC$ and their reachability}
Place the numbers $1, \dots, n$  clockwise on the boundary of the disc. 
Let $\sigma$ be a decorated permutation. 
That is, $\sigma\in \sn$ with the fixed points coloured in two colours \cite[Def. 13.3]{Pos}, 
that is, $1$ and $-1$.
We say that a Postnikov diagram or a plabic graph is of \emph{type} $\sigma$ if 
the corresponding necklace is defined by the permutation $\sigma$. For 
instance, a plabic graph with face labels a maximal collection of non-crossing $k$-sets 
has type $\sigma_{k, n}$, where $\sigma_{k, n}\colon i \mapsto i+k, \forall i \in [n]$ 
(and $[n] = \mathbb{Z}_n$). The boundary faces of a plabic graph are labelled by the 
associated necklace, $I_1, \dots, I_n$. 
The boundary region touching the arc  from $i-1$ to $i$ in clockwise order is labelled by $I_i$, and if 
this region is also bounded by the left arc incident at $j$, then $I_i=I_j$.

Let $G$ be a plabic graph of type $\sigma$, and $\maxNC=\mathcal{F}(G)$, the set of face labels for $G$. 
Then
$\maxNC$ is a maximal collection of non-crossing $k$-sets in the positroid $\positr$ containing $\nlace$
\cite[Thm. 1.5]{OPS}.  Thus
$M_{\maxNC}$ is a cluster tilting object in $\GP B$,
following  Theorem~\ref{main1}.

\begin{proposition} \label{prop:reachability}
 Any cluster tilting object of the form $M_{\maxNC'}\in \GP B$ is reachable from $M_{\maxNC}$. 
\end{proposition}

\begin{proof}
Let $\maxNC'$ be a maximal collection of non-crossing $k$-sets 
in $\positr$ containing $\nlace$. 
By \cite[Thm. 1.5]{OPS}, $\maxNC'=\mathcal{F}(G')$ 
for a  plabic graph $G'$, and
by \cite[Thm. 13.4]{Pos}, $G'$ can be obtained from $G$
by a sequence of three types of moves, (M1)-(M3) (see \cite{OPS, Pos}). 
Moves (M2) and (M3) do not change the number of faces and their labels,  
while (M1) is the square move,
which corresponds to a 
Pl\"{u}cker relation and to the mutation of the cluster tilting object.
See Figure \ref{fig:sqmstrands} for an illustration, where the square move replaces $Lbd$ by $Lac$, and 
the corresponding mutation sequences are: 
\[
\xymatrix{
0\ar[r] & M_{Lac}\ar[r]& M_{Lad}\oplus M_{Lbc} \ar[r] &M_{Lbd} \ar[r] &0}\]
and
\[
\xymatrix{
0\ar[r] & M_{Lbd}\ar[r]& M_{Lab}\oplus M_{Lcd} \ar[r] &M_{Lac} \ar[r] &0. 
}
\]
Therefore $M_{\maxNC'}$ is reachable from $M_{\maxNC}$.   
\end{proof}

\begin{figure}
\begin{tikzpicture}[scale=0.6,baseline=(bb.base),
  strand/.style={teal,dashed,thick},
  quivarrow/.style={blue, -latex, thick},
  doublearrow/.style={black, latex-latex, very thick}]
\newcommand{\strarrow}{\arrow{angle 60}}
\newcommand{\dotrad}{0.1cm} %
\newcommand{\bdrydotrad}{{0.8*\dotrad}} %
\path (0,0) node (bb) {}; %

\draw [strand] plot coordinates {(4,-2) (1.8,-0.2) (0.8,0.8) (-0.2, 1.8) (-2,4)}[postaction=decorate,decoration={markings,
mark= at position 0.17 with \strarrow,
mark= at position 0.5 with \strarrow,
mark= at position 0.83 with \strarrow}];

\draw [strand] plot coordinates {(-4,2) (-1.8,0.2) (-0.8,-0.8) (0.2,-1.8) (2,-4)}[postaction=decorate,decoration={markings,
mark= at position 0.17 with \strarrow,
mark= at position 0.5 with \strarrow,
mark= at position 0.83 with \strarrow}];

\draw [strand] plot coordinates {(-4,-2) (-1.8,-0.2) (-0.8,0.8) (0.2,1.8) (2,4)}[postaction=decorate,decoration={markings,
mark= at position 0.17 with \strarrow,
mark= at position 0.5 with \strarrow,
mark= at position 0.83 with \strarrow}];

\draw [strand] plot coordinates {(4,2) (1.8,0.2) (0.8,-0.8) (-0.2,-1.8) (-2,-4)}[postaction=decorate,decoration={markings,
mark= at position 0.17 with \strarrow,
mark= at position 0.5 with \strarrow,
mark= at position 0.83 with \strarrow}];

\draw (2.3,4.3) node {$a$};
\draw (-2.3,-4.3) node {$c$};
\draw (2.3,-4.3) node {$b$};
\draw (-2.3,4.3) node {$d$};

\draw (0,0) node (ac) {$Lbd$};
\draw (0,3.5) node (ab) {$Lab$};
\draw (3.5,0) node (bc) {$Lbc$};
\draw (0,-3.5) node (cd) {$Lcd$};
\draw (-3.5,0) node (ad) {$Lad$};

\draw [quivarrow] (ad)--(ac);
\draw [quivarrow] (bc)--(ac); 
\draw [quivarrow] (ac)--(cd); 
\draw [quivarrow] (ac)--(ab);

\begin{scope}[shift={(12,0)}]

\draw [strand] plot[smooth] coordinates {(4,-2) (2,-1.8) (0,-1.3) (-1.3,0) (-1.8,2) (-2,4)}[postaction=decorate,decoration={markings,
mark= at position 0.1 with \strarrow,
mark= at position 0.3 with \strarrow,
mark= at position 0.5 with \strarrow,
mark= at position 0.74 with \strarrow,
mark= at position 0.92 with \strarrow
}];

\draw [strand] plot[smooth] coordinates {(-4,2) (-2,1.8) (0,1.3) (1.3,0) (1.8,-2) (2,-4)}[postaction=decorate,decoration={markings,
mark= at position 0.1 with \strarrow,
mark= at position 0.3 with \strarrow,
mark= at position 0.5 with \strarrow,
mark= at position 0.74 with \strarrow,
mark= at position 0.92 with \strarrow}];

\draw [strand] plot[smooth] coordinates {(-4,-2) (-2,-1.8) (0,-1.3) (1.3,0) (1.8,2) (2,4)}[postaction=decorate,decoration={markings,
mark= at position 0.1 with \strarrow,
mark= at position 0.3 with \strarrow,
mark= at position 0.5 with \strarrow,
mark= at position 0.74 with \strarrow,
mark= at position 0.92 with \strarrow}];

\draw [strand] plot[smooth] coordinates {(4,2) (2,1.8) (0,1.3) (-1.3,0) (-1.8,-2) (-2,-4)}[postaction=decorate,decoration={markings,
mark= at position 0.1 with \strarrow,
mark= at position 0.3 with \strarrow,
mark= at position 0.5 with \strarrow,
mark= at position 0.74 with \strarrow,
mark= at position 0.92 with \strarrow}];

\draw (2.1, 4.4) node {$a$};
\draw (-2.1, -4.4) node {$c$};
\draw (2.1, -4.4) node {$b$};
\draw (-2.1, 4.4) node {$d$};

\draw (0,0) node (bd) {$Lac$};

\draw (0,3.5) node (ab) {$Lab$};
\draw (3.5,0) node (bc) {$Lbc$};
\draw (0,-3.5) node (cd) {$Lcd$};
\draw (-3.5,0) node (ad) {$Lad$};

\draw [quivarrow] (bd)--(ad);
\draw [quivarrow] (bd)--(bc); 
\draw [quivarrow] (cd)--(bd); 
\draw [quivarrow] (ab)--(bd);

\draw [quivarrow] (ad)--(ab);
\draw [quivarrow] (bc)--(ab); 
\draw [quivarrow] (ad)--(cd);
\draw [quivarrow] (bc)--(cd); 
\end{scope}

\draw (5.5, 0) node {$\leftrightsquigarrow$};

\end{tikzpicture}
\caption{A square move and the associated quiver mutation 
(cf. \cite[Fig.~2]{JKS2})}
\label{fig:sqmstrands}
\end{figure}
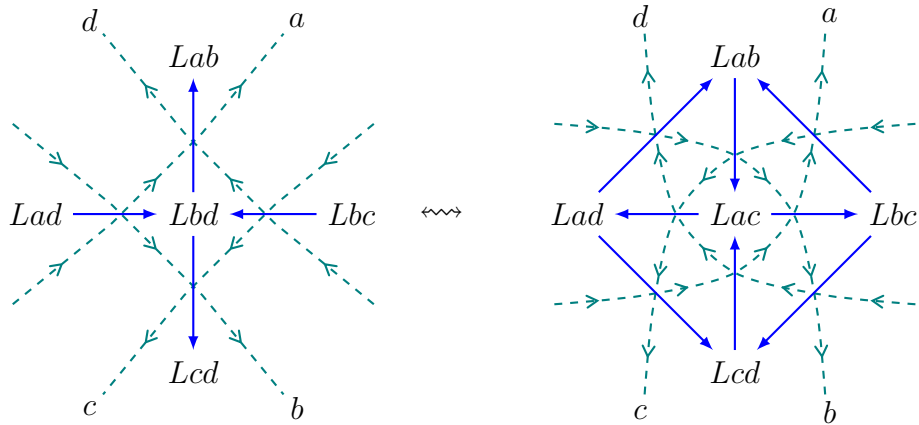

From now on, 
in $\GP B$, by {\it reachable} we mean summands of a cluster 
tilting object that is reachable from the cluster tilting object of the form 
$M_{\maxNC}$, where $\maxNC$ is a maximal collection of non-crossing sets 
corresponding to modules in $\GP B$. 

\subsection{Comparing endomorphism algebras}
For any $M_I, M_J\in \CM C$, which are graded $\ring$-modules, we have  \[\Hom(M_I, M_J)\cong \ring.\]
Let $f_{I, J}$ be a generator of $\Hom(M_I, M_J)$, viewed as a graded $\ring$-module.
Such a generator $f_{I, J}$ is unique up to a scalar and there exists a vertex $i$ such 
that the restriction of $f_{I, J}$ to that vertex induces an isomorphism from $(M_I)_i$ 
to $(M_J)_i$, and in this case, $I\leq_{i+1}J$.

We denote by $IS$ the union of any two sets $I, S\subseteq [n]$.

\begin{lemma}\label{lem:EndTech1}
Let $I, J, K$ be $k$-sets. Assume that $S\subset [n]$  is disjoint to these three 
$k$-sets. Then
$f_{I, J}\colon M_I \to M_J$ factors through $M_L$ if and only if 
$f_{IS, JS}\colon M_{IS} \to M_{JS}$ factors through $M_{LS}$.
\end{lemma}
\begin{proof}

Suppose that $f_{I, J}$ factors through $M_L$, that is, 
$f_{I, J}=gh$ for some $g\colon M_L\to M_J$ and $h\colon M_I\to M_L$, 
 and 
the restriction of $f_{I, J}$ to vertex $i$ is an isomorphism. As $M_I$ and $M_J$
are Cohen-Macaulay modules of rank one,  any nonzero 
homomorphism between them is injective. Therefore, the 
restrictions of $g$ and $h$ to vertex $i$ are also isomorphisms, and so  
\begin{equation}\label{eq:orderati}
I\leq_{i+1} L \leq_{i+1} J.
\end{equation}
Conversely, if the three sets satisfy \eqref{eq:orderati}, 
then $f_{I, J}=f_{L, J}f_{I, L}$.

Now the relations in \eqref{eq:orderati} are preserved by adding elements that 
are not contained in $I\cup J\cup L$, or by removing elements in $I\cap J\cap L$.
So the lemma follows.
\end{proof}

\begin{corollary}\label{lem:EndTech2}
Let $\maxNC$ be a collection of $k$-sets and let $I\subset[n]$ be disjoint to any $J\in \maxNC$.
Denote the set $\{IJ\colon J\in \maxNC\}$ by $\maxNC I$. 
Then,    
as $\ring$-algebras, 
$$\End M_{\maxNC} \cong \End M_{\maxNC I}.$$ 
 \end{corollary}

\subsection{Gabriel quivers, plabic graphs and associated cluster algebras}\label{sec:6.3} 
Let $\maxNC$ be any maximal collection of non-crossing sets in $\positr$ containing $\nlace$. 
Let $X=\{\Delta_I\colon I\in \maxWS\}$. 
Extend $\maxNC$ to a maximal collection $\tilde{\maxNC}$ of non-crossing $k$-sets. 
Then $\{\Delta_I\colon I\in \tilde{\maxWS}\}$ is a cluster in $\CC[\Gr(k, n)]$ \cite{Scott}, containing $X$. 
 In particular, $X$ is algebraically independent.

We denote  by $Q_{\maxWS}$ the Gabriel quiver of $(\End M_{\maxNC})^{\text{op}}$,
instead of the usual $Q_{M_\maxNC}$, in order to simplify the notation. Let 
$B_{\maxWS}$ be the exchange matrix defined by $Q_{\maxNC}$. 
As mutations in $\GP B$ are mutations in $\CM C$ (see Theorem \ref{main1}) and 
$\CM C$ provides a categorification of the cluster structure on $\CC[\Gr(k, n)]$ \cite{JKS1}, 
$(X, B_\maxNC)$ determines a cluster subalgebra of $\CC[\Gr(k, n)]$, denoted by $\clalg$.
By Proposition \ref{prop:reachability}, $\clalg$ is independent of the choice of $\maxNC$.
The cluster variables of $\clalg$ are the cluster characters of reachable indecomposable rigid 
objects in $\GP B$.

Let $\sigma$ be the decorated permutation associated to the necklace $\nlace$. 
We recall definitions and basic results concerning connected components of decorated permutations,  Grassmann necklaces and positroids \cite[\S 5]{OPS}. Let 
\[ [n]=\set_1\cup \dots \cup\set_t\] 
be a partition of $[n]$ into a union of disjoint subsets. We say the partition is \emph{non-crossing} if any two distinct sets $\set_i$ and $\set_j$ are non-crossing.

Assume the partition is  non-crossing such that if $a\in \set_j$, then $\sigma(a)\in \set_j$, $\forall j$.
Let %
\begin{equation}\label{eq:sigmaj}
\sigma_j=\sigma|_{\set_j}
\end{equation}
and let $\nlace_{j}$ be the associated Grassmann necklace of $\sigma_{j}$ on the set $\set_j$.  
If the partition is the finest possible,  
then $\sigma_{j}$ is called a \emph{connected component} 
of $\sigma$ and $\nlace_{j}$ is called a \emph{connected component} of $\nlace$. 

The permutation $\sigma$ and Grassmann necklace $\nlace$ are said to be \emph{connected} if they have exactly one connected component, and in this case we also say that the \emph{algebra $B$  is connected}. The necklace $\nlace$ is connected if and only if $I_1, \dots, I_n$ are all distinct
\cite[Lem. 5.6]{OPS}.

Assume that $\sigma$ is disconnected. By \cite[Lem. 5.6]{OPS} and its proof,  there exist $i\not=j$ such that $I_i=I_j$
and $\sigma$ maps the intervals $[i, j)$ and $[j, i)$  to themselves. 
Let \[\set_1=[i, j) \text{ and } \set_2=[j, i),\] which form a non-crossing partition of $[n]$. 

For any $J\in \positr$, let
\[J^1=[i, j)\cap J  \text{ and } J^2=[j, i)\cap J,  \]
and 
\[
n_1=|[i, j)|, n_2=|[j, i)|, k_1=|I_i^1| \text{ and } k_2=|I_i^2|.
\]
Then $n_1+n_2=n$, $k_1+k_2=k$, and 
\[\nlaceij=\{I^1_i, \dots, I^1_{j-1}\} \text{ and }  \nlaceji=\{I^2_j, \dots, I^2_{i-1}\}\]
are necklaces on the sets $\set_1$ and $\set_2$ \cite[Prop. 5.10]{OPS}. Let
$\positrij$ and $\positrji$ be 
the positroids determined by $\nlaceij$ and $\nlaceji$, respectively.
Note that these are necklaces and 
positroids for the smaller Grassmannians, $\Gr(k_1, n_1)$ and $\Gr(k_2, n_2)$.

By \cite[Prop. 5.8 and Prop 5.10]{OPS}, $\positr$ is the \emph{direct sum} of 
$\positrij$ and $\positrji$, in the sense that, $J\in \positr$ if and only if 
$J^1\in \positrij$ and $J^2\in \positrji$. Moreover, 
$\maxNC$ is a collection of non-crossing $k$-sets if and only if there are collections $\maxNCij$, $\maxNCji$ 
of non-crossing $k_1$-sets in $\positrij$ and non-crossing $k_2$-sets in  $\positrij$, respectively, such that 
\begin{equation}\label{eq:divS}
\maxNC=\{J\cup I_i^2\colon J\in \maxNCij\} \cup
\{I_i^1\cup J\colon J\in \maxNCji\}.
\end{equation}
In particular, for any $s\in [i, j)$ and $t\in [j, i)$,  
\begin{equation}\label{eq:cols}
I_s=J\cup I_i^2 \text{ for some } J\in \positrij \text{ and }
I_t= I_i^1\cup J \text{ for some } J\in \positrji.
\end{equation}
Denote the sets $\{J\cup I_i^2\colon J\in \maxNCij\}$ and 
 $\{ I_i^2\cup J\colon J\in \maxNCji\}$ in \eqref{eq:divS}  by 
$\maxNCone$ and $\maxNCtwo$, respectively.
 
\begin{theorem}\label{thm:locacy}
Let $G$ be a plabic graph of type $\sigma$ and $\maxNC=\mathcal{F}(G)$. Then 
 \begin{equation}\label{eq:compareQ}Q_\maxNC^{\circ}=Q_G.\end{equation}
Consequently, as cluster algebras,  $\clalg\cong \algMS$, and thus $\clalg$ is locally acyclic.
\end{theorem}
\begin{proof} Let $D$ be the Postnikov diagram associated to $G$.

First, consider the case where $D$  is connected. Let $D'$ be the Postnikov diagram
obtained from $D$ by reversing the directions of the strands in $D$.
Compare the face labels $\maxNC$ for $D$
following the left-target convention as in \cite{MS}, 
and the face labels $\maxNC'$ for $D'$ 
following the left-source convention as in \cite{CKP}. We have
\[\maxNC=\{I\colon [n]\backslash I\in \maxNC'\}.\]
Note that $D'$ is also  connected and of type 
$\sigma^{-1}$. The associated plabic graph $G'$ is obtained from $G$ by switching the colours of the internal vertices, 
and so using the notation in \cite{MS}, 
\begin{equation}\label{eq:quivfromG1}
Q_G=Q_{G'}^{\text{op}}.
\end{equation}
Canakci--King--Pressland proved that $A_{D'}\cong (\End M_{\maxNC'})^{\text{op}}$ \cite[Cor. 8.3]{CKP}, where $A_{D'}$ is the dimer algebra defined by 
 the quiver $Q_{D'}$ constructed from $D'$, subject to some (admissible) relations. 
The isomorphism implies that  
\begin{equation}\label{eq:quivfromG2}
Q_{\maxNC'}=Q_{D'}. 
\end{equation}
The constructions of the quiver associated to 
a Postnikov diagram 
(or equivalently, the associated plabic graph) given by Canakci--King--Pressland and by
Muller-Speyer are the same (see \cite[p. 7]{MS} and \cite[Def.~2.8]{CKP}), with one exception that, 
 in \cite{CKP}, $Q_{D'}$ includes arrows between frozen vertices, whereas in \cite{MS} such arrows are omitted. So 
\begin{equation}\label{eq:quivfromG3}
Q_{D'}^\circ \text{ from \cite{CKP} is the same as } Q_{G'} \text{ from \cite{MS}}.
\end{equation}
Note that the algebra $C=C_{k, n}$ depends of $(k, n)$. Analogously, we have an algebra $C_{n-k, n}$, where 
the relation $x^k=y^{n-k}$ in $C$ is replaced by $x^{n-k}=y^k$.
Using the duality $\Hom_\ring(-, \ring)\colon \CM C\to \CM C^{\text{op}}$ and the isomorphism 
$C^{\text{op}} \cong C_{n-k, n}$, we have 
\begin{equation}\label{eq:quivfromG4}
\End M_{\maxNC}\cong (\End M_{\maxNC'})^{\text{op}} \text{ and thus } Q_\maxNC=Q_{\maxNC'}^{\text{op}}.
\end{equation}
Therefore, by \eqref{eq:quivfromG1}-\eqref{eq:quivfromG4}, we have the required equality for the connected case,
\[
Q_{\maxNC}^\circ =Q_G.
\]

Next, assume that $D$ is not connected and that \eqref{eq:compareQ} holds for smaller plabic graphs. 
Now the necklace 
$\nlace$ is not connected either.  So there exists $i\not= j$ such that $I_i=I_j$. 
We know that $\sigma$ maps $[i, j)$ to itself. 
Let $\sigma^1$ be the decorated permutation on $[n]$ 
such that  $\sigma^1|_{[i, j)}=\sigma|_{[i, j)}$ and  fixes all the points in $[j, i)$, 
where points in $I_i^2$ are coloured by $-1$ and the other 
points in $[j, i)$ are coloured by 1. Note that in the Postnikov diagram $D$, a clockwise loop is attached to 
each fixed point coloured by $-1$, and an anti-clockwise loop is attached to each fixed points coloured by 1. 
This colouring of the fixed points is determined by the property of $\nlace$ in \eqref{eq:cols}.
Therefore, $\sigma^1$ is also 
a decorated permutation of type $(k, n)$, with the corresponding necklace 
$\nlace_{\sigma^{1}}=\{I_i, I_{i+1}, \dots, I_j, I'_{j+1}, \dots, I'_{i-1}\}$, where 
 $I_l'=I_j$ for all $l\in [j+1, i-1]$. Let $\sigma^2$ and $\nlace_{\sigma^2}$
be defined analogously. The decorated permutation $\sigma^j$ should not be 
confused with the decorated permutation $\sigma_j$ defined in \eqref{eq:sigmaj}.

Let $B_1$, $B_2$ be the algebras associated to the necklaces $\nlace_{\sigma^{1}}$
and $\nlace_{\sigma^{1}}$, respectively. Use the same notation as in 
Theorem \ref{main1}, where the categories $\catD$ and  $\catD_2$ are defined relative to $\catD_1=\GP B_1$. 
Then $M_{\maxNCone}$ is a cluster tilting object in $\GP B_1$ and 
$M_{\maxNCtwo}\in \catD_2$.  By Theorem \ref{main1}, $Q_{\maxNCone}^\circ$ is a full subquiver 
of $Q_{\maxNC}^\circ$ and
 there are no arrows between  mutable vertices in $Q_{\maxNCone}$ and 
vertices in $Q_{\maxNC}\backslash Q_{\maxNCone}$.
Similarly,  $Q_{\maxNCtwo}^\circ$ is a full subquiver 
of $Q_{\maxNC}^\circ$ and
 there are no arrows between  mutable vertices in $Q_{\maxNCtwo}$ and 
vertices in $Q_{\maxNC}\backslash Q_{\maxNCtwo}$. 
Note that $M_{I_i}$ is a summand of $B_1$, $B_2$ and $B$.  The vertices corresponding to $M_{I_i}$
are frozen in all the three Gabriel quivers, $Q_\maxNC$, $Q_{\maxNCone}$ and $Q_{\maxNCtwo}$. Therefore, 
\begin{equation}\label{eq:Gabquiv}
Q_{\maxNC}^\circ = Q_{\maxNCone}^\circ \cup Q_{\maxNCtwo}^\circ.
\end{equation}

The alternating face labeled by 
$I_i$  divide  $D$ into two Postnikov diagrams $D_1$ and $D_2$ on $[i, j)$ and $[j, i)$, respectively.
Let $G_1$, $G_2$ be associated  plabic subgraphs  of $G$. Then $D_1$ and $G_1$ are
of type $\sigma_1$, while $D_2$ and $G_2$ are
of type $\sigma_2$.
The subgraphs $G_1$ and $G_2$ share a common vertex labeled by $I_i$ and 
\begin{equation}\label{eq:plabicg}
Q_G=Q_{G_1}\cup Q_{G_2}.
\end{equation}
Note that the construction of each $Q_{G_i}$ is independent 
of whether $G_i$ is regarded as subgraph of $G$ or considered in isolation, and 
by induction, 
\begin{equation}\label{eq:plabicg2}
Q_{\maxNCij}^\circ =Q_{G_1} \text{ and } Q_{\maxNCji}^\circ =Q_{G_2}.
\end{equation}
On the other hand, 
by definition, $I_i^2\subseteq [j, i)$ and $J\subseteq [i, j)$ for any $J\in \positrij$. Therefore, $I_i^2$ 
 is disjoint to any $J\in \positrij$ and thus to any $J\in \maxNCij$ . So,
by Lemma \ref{lem:EndTech1}, \[\End M_{\maxNCone} \cong \End M_{\maxNCij},\] which implies
$Q_{\maxNCone}=Q_{\maxNCij}$. 
Similarly, 
$Q_{\maxNCtwo}=Q_{\maxNCji}$. 
Therefore, following  \eqref{eq:Gabquiv}-\eqref{eq:plabicg2}, we have the required equality,
\[
Q_{\maxNC}^\circ = Q_G. 
\]
Consequently, $\clalg\cong \algMS$. 
Finally, by \cite[Thm 3.3]{MS},  $\algMS$ is locally acyclic, and so  
$\clalg$ is also locally acyclic.
\end{proof}

\section{Grassmannian cluster subalgebras and positroid varieties}
For a subcategory $\catC$ of $\CM C$, we denote by $\mathsf{A}_{\catC}$ 
the subspace of $\CC[\Gr(k, n)]$ spanned by $\Psi_{M}$ for all $M\in \catC$.
We have 
\[
\algGP\subseteq \algCM\subseteq \CC[\Gr(k, n)]
\text{ and } \algGP, \algDd\subseteq \algD\subseteq \CC[\Gr(k, n)], 
\]
which are all subalgebras, as the corresponding subcategories of $\CM C$
contain the zero object and are closed under taking direct sums. These subcategories
all contain $B$ as well. 
Denote by $(\mathsf{A}_{\catC})_{B}$ the localisation of $\mathsf{A}_{\catC}$ at
$\{\Psi_{B'}\colon B'\in \add B\}$. By \cite[Thm. 9.5]{JKS1}, $\Psi$
defines a one-to-one correspondence between reachable rigid modules in $\CM C$
and the cluster variables of the cluster algebra on $\CC[\Gr(k, n)]$, constructed by 
Scott \cite{Scott}. Therefore,  
 \[\mathsf{A_{CMC}}=\CC[\Gr(k, n)],\] 
 and so by  Proposition \ref{prop:mainproperty},
\begin{equation}\label{eq:algDCGratB}
(\algD)_{B}=\CC[\Gr(k, n)]_B.
\end{equation}

\subsection{Bases for subalgebras of $\CC[\Gr(k, n)]$}
As before, let $U$ and $T=U\oplus V$ be cluster tilting objects in $\GP B$ and  
$\CM C$, respectively.  We assume that  $V$ does not have any indecomposable direct summand in $\add B$.
Recall the notation, $A=A_T=(\End T)^{\mathrm{op}}$, 
$A_U=(\End U)^{\mathrm{op}}$,
and $Q$ and $Q_U$ are the Gabriel quivers of $A$ and 
$A_U$, respectively. Following Theorem \ref{main1} (4), we have $Q_U^\circ\subseteq Q^\circ$.

For $X\in \add T$, let 
 \[\supp(X)=\{i\colon T_i\in \add X\},\]
where, as before, $\oplus_i T_i$ is the basic cluster tilting object determined by $T$.
For $d=(d_i)\in K(\CM A)$, written as a coordinate vector 
with respect to the basis $[T, T_i]$, let 
\[\supp(d)=\{i\colon d_i\not=0\}.\] 

The result below follows from Proposition \ref{mutgpBC}.  

\begin{lemma} \label{lem:gcogv}
Let $M\in \CM B$ and $N\in \catD_2$.
\begin{enumerate}
\item $\supp(\gv_M)\subseteq \supp(U)$.

\item If $M\in \GP B$, then $\supp(\cogv_M)\subseteq \supp(U)$. 

\item If $N\in \catD_2$, then $\supp(\gv_N), ~ \supp(\cogv_N)\subseteq \supp(V\oplus B)$. 
\end{enumerate}
\end{lemma}

Let $\bas$ be the generic basis of $\CC[\Gr(k, n)]$. In particular, it is a basis as desribed in  
Proposition \ref{Thm:GrassGen-gv-cogv}.  Let  
\[
\bas_\catC=\bas\cap \mathsf{A}_{\catC}.
\]

\begin{theorem} \label{thm:genericbasis}
The sets $\basGP$, $\basD$ and $\bas_{\catD_2}$ are bases of the
algebras $\algGP$, $\algD$ and $\algDd$, respectively. Their elements have pairwise 
distinct minimal and distinct maximal exponents.  
\end{theorem}
\begin{proof}
First, consider the case where $M\in \catD$. Express $\Psi_M$ as a linear 
combination using~$\bas$, 
\begin{equation*}\label{eq:lincombBD}
\Psi_M=\sum_{i=1}^s a_{i} \Phi_{N_i}, \text{ where } a_i\not=0.
\end{equation*}
By Proposition \ref{prop:CompExt}, for any $i$, 
 \[
 \dimv\Ext^1(T, N_i)\leq \dimv \Ext^1(T, M).
 \] 
As $M\in \catD$,  $\Ext^1(B, M)=0$. Therefore,
\begin{equation}\label{eq:newext}\Ext^1(B, N_i)=0, \end{equation}  
and so $N_i\in \catD$. Consequently, 
$\basD$ is a basis of $\algD$.

Second, let $M\in \GP B$. We first show that each $N_i\in \CM B$.
Write $U/B$ for the complementary summand of $B$ in $U$, that is, $(U/B)\oplus B=U$.
By Proposition \ref{prop:GPBV} and the assumption that $M\in \GP B$, 
\[
   \Ext^1(T, M)=\Ext^1(U, M) =\Ext^1(U/B, M), 
\]
which is only supported at the mutable vertices of $Q_U$. Note that 
 \eqref{eq:betaUT} implies  
\[
\supp(\beta([S_i]))\subseteq \supp(U),
\] 
for any mutable vertex $i\in Q_U$. Hence, for any $d\leq \dimv \Ext^1(T, M)$, 
\[
\supp(\beta(d))\subseteq \supp(U),
\]
which together with Lemma \ref{lem:gcogv} and Proposition \ref{prop:CompExt} 
implies that 
\begin{equation}\label{eq:suppNiU}
\supp(\gv(N_i)) \subseteq \supp(U).
\end{equation}
By \cite[Lem. 12.5]{JKS3}, 
$N_i$ has a \emph{generic} $\add T$-presentation 
\begin{equation*}
\xymatrix{0\ar[r]& T''\ar[r] &  T'\ar[r] &  N_i \ar[r] & 0,}
\end{equation*}
where $T'$ and $T''$ have no common summands. We have 
\[
 \gv (N_i)= [T, T']-[T, T''], 
\]
and so by \eqref{eq:suppNiU},  $T', ~T'' \in \add U$. Consequently,  $N_i$ is a $B$-module and so 
 $N_i\in~\CM B$, as required.
 
As $M\in \GP B\subseteq \catD$, by the symmetry in $\Ext^1(-, -)$ and \eqref{eq:newext}, 
we have
\[
\Ext^1(N_i, B) \cong \Ext^1(B, N_i)=0, 
\]
where the extensions are extensions of $C$-modules, and in fact by \eqref{eq:extCB}, 
\[\Ext^1_C(N_i, B)=\Ext^1_B(N_i, B).\]
Therefore, $N_i\in \GP B$ and thus $\basGP$ is a basis of $\algGP$. \\

Third, consider the case where $M\in \catD_2$. Using Corollary 
 \ref{cor:mutinD2} and similar to \eqref{eq:betaUT}, we have that 
$\beta([S_i])\in \supp(B\oplus V)$ for any simple $S_i$ corresponding to 
a mutable summand of $T$ in $\add V$. Note that $\Ext^1_C(T, M)=\Ext^1_C(V,M)$, 
by Proposition \ref{prop:GPBV}, and so for any $0\leq d\leq \dimv \Ext^1_C(T, M)$, 
\[
\supp(\beta(d))\subseteq \supp(B\oplus V).
\]
Together with  Lemma \ref{lem:gcogv} (3), this implies $\supp(\gv(N_i))\subseteq \supp(B\oplus V)$. As above, $N_i$ has a generic 
$\add T$-presentation 
\begin{equation*}
\xymatrix{0\ar[r]& T''\ar[r] & T'\ar[r] &  N_i \ar[r] & 0}
\end{equation*}
with $T', T''\in \add (B\oplus V)$, and 
\[
 \gv (N_i)= [T, T']-[T, T''].
\]
Let $N_{ij}$ be an indecomposable summand of $N_i$ not in $\catD_2$. Then $N_{ij}\in \GP B\backslash 
\add B$ by Proposition \ref{prop:tem} (1) (and Lemma \ref{lem:elabtem}). The composition $T'\rightarrow N_i\rightarrow N_{ij}$ 
is an $\add T$-approximation. It follows that the minimal $\add T$-approximation of 
$N_{ij}$ has no summands from $\add U/B$. However, since $N_{ij}\in \GP B\backslash \add B$, the kernel of the approximation must contain summands from $\add U/B$. This is a contradiction, since $\gv(N_i)=\sum_j \gv(N_{ij})\in \supp(B\oplus V)$. 
So $N_i\in \catD_2$.

Finally, the basis elements of all three bases inherit the property of having pairwise distinct minimal and maximal exponents from 
those in $\bas$.
\end{proof}

\begin{remark}
We remark that  $\basCM$ is in general not a basis for $\algCM$. Indeed, suppose 
that $\GP B$ is a proper subcategory of $\CM B$ and 
let $M\in \CM B\backslash\GP B$. Following Proposition \ref{prop:mainproperty} and Theorem \ref{thm:genericbasis}, we have 
\begin{equation}\label{eq:bcmb}
\Psi_M\Psi_{B^r}=\sum_{X\in \catD} a_X\Psi_X =\sum_{\Psi_N\in \bas} a_N\Psi_N,
\end{equation}
where each $X$ with $a_X\not=0$ is an extension of $M$ and $B^r$. 
As $\Ext^1_C(M, B)=\Ext^1_B(M, B)$ (see \eqref{eq:extCB}),  any extension of $M$ by $B$ is 
a $B$-module. On the other hand, 
as $M\not\in \GP B$, any non-trivial extension of $B$ by $M$ is not a $B$-module. 
So there exists $X\not\in \CM B$ in \eqref{eq:bcmb} with $a_X\not=0$. 
Consequently,  $X\not\in \GP B$ and thus 
 there exists $N\not\in \GP B$ in \eqref{eq:bcmb} with $a_N\not=0$.
This implies that $N$ has a direct summand 
in $\catD_2\backslash \add B$ and $\Psi_N\not\in \basCM$. 
Therefore,  $\basCM$ is not a basis for $\algCM$ in this case. 
\end{remark}

\subsection{An ideal of $\algD$}
When  $\GP B=\CM C$ (i.e. $B=C$), we have $\Pi=\Gr(k, n)$. In this 
case $\CM C$ equipped with $\Psi$ gives a categorification of the 
Grassmannian cluster algebra $\CC[\Gr(k, n)]$. So for the remainder of this 
paper, we assume that $B\not=C$.%

Let $\ideal$ be the ideal of $\algD$ generated by $\{\Psi_N\colon N\in \catD_2\backslash \add B\}$, 
which is nonempty, as it contains  $\add C\backslash \add B$.
Let $\bas_{\ideal}=\bas\cap \ideal$.
\begin{proposition}\label{prop:ideal} We have the following.
\begin{enumerate}
\item\label{itm:idea1}
$ \bas_{\ideal}=\{\Psi_M\Psi_N\colon \Psi_M\in \basGP, ~ \Psi_N\in \bas_{\catD_2} 
\text{ and } N \in  \catD_2\backslash \add B\}$ and this is a basis of $\ideal$.
\item\label{itm:idea2} $\ideal\cap \algGP=0$.

\item\label{itm:idea3} $\algGP\cong \algD/\ideal$.
\end{enumerate}
\end{proposition}
\begin{proof} 
(1) Clearly the right hand side  is contained in and spans the ideal $I$. Consider the 
product $\Psi_M\Psi_N$ with $\Psi_M\in \basGP$ and $\Psi_N\in \bas_{\catD_2}$. 
By Proposition \ref{prop:GPBV}, $\Ext^1(M,N)=\Ext^1(N,M)=0$. So, 
by \cite[Thm. 1.2]{CS} (see also \cite[Thm. 1.1]{GLS05}),
$\Psi_{M\oplus N}=\Psi_M\Psi_N$ is an element of $\bas$. Hence, 
the right hand side is contained in $\bas$. So (1) holds.

(2) Each element in the basis $\bas_\ideal$ has a factor 
that does not belong to $\algGP$. So $\basGP\cap \bas_\ideal=\emptyset$. As they 
are both a subset of the basis  $\bas$, their union is linearly independent. 
So $\ideal\cap \algGP=0$.

(3) 
For any $X\in \catD$, we have $X=M\oplus N$ for some 
$M\in \GP B$,  $N\in \catD_2$, and 
\begin{equation*}\label{eq:charprod}\Psi_{X}=\Psi_M \Psi_N.\end{equation*}
Therefore, by Theorem \ref{thm:genericbasis}, for any $X\in \catD$, 
we have 
\begin{equation}\label{eq:PsiXD}
\Psi_X=\sum_{X_i\in \GP B} a_i\Psi_{X_i}+\sum_{Y_i\not\in \GP B} b_i\Psi_{Y_i}
\end{equation}
with the second sum contained in $\ideal$. So the composition
$\algGP\to \algD/\ideal$ is surjective. On the other hand, by (2), the projection is also injective. Therefore 
$\algGP\cong\algD/\ideal$. 
\end{proof}

\subsection{A generalised partition function}
We recall the definition of the generalised partition function defined in \cite{JKS3} 
(see also \cite{Pr2} and \cite{CKP}) and discuss some simple properties.
Let $M\in \CM B$ and consider the partial projective $B$-presentation
\begin{equation*}
\xymatrix{
0\ar[r]& \Omega_M\ar[r] & P_M\ar[r] & M\ar[r] &0.
}
\end{equation*}
Then $\Omega_M\in \GP B$ (see Lemma \ref{lem:GPBBapp}). Define 
\begin{equation}\label{eq:defptf}
\ptf_M=\frac{\Psi_{\Omega_M}}{\Psi_{P_M}}.
\end{equation}

Let $U$ be a cluster tilting object in $\GP B$ and $U_i$ be a mutable summand of $U$. 
Let $\Omega_{U_i}$ be the minimal syzygy of $U_i$. Define 
\[U^\Omega=B\oplus (\oplus_i \Omega_{U_i}).
\]
Similarly, we can define $U^{\Sigma}$, replacing syzygies by cosyzygies.

\begin{proposition} Let $\maxWS$ be a maximal collection of non-crossing sets in the positroid $\positr$ associated to $B$.
\begin{itemize}\item[(1)] $U^\Omega$ and $U^{\Sigma}$ are  cluster tilting objects in $\GP B$.
\item[(2)]  $(U_i, U_i^*)$ is a mutation pair in $\GP B$ if and only if  
so is $(\Omega_{U_i}, \Omega_{U_i^*})$.
\end{itemize}
Consequently, $U$ is reachable from $M_{\maxWS}=\oplus_{I\in \maxWS} M_I$ if and only $U^{\Omega}$ (resp. $U^{\Sigma}$)
is reachable from $M^{\Omega}_{\maxWS}$ (resp. $M^{\Sigma}_{\maxWS}$).
\end{proposition}
The proof of  \cite[Prop 9.5]{JKS3} can be adapted to this setting and we skip the details.

\begin{proposition}\label{prop:chsum} 
 The map $\ptf\colon \GP B\to \CC[\opoid]$ is a cluster character. 
\end{proposition}
\begin{proof} 
By the description of $\CC[\opoid]$ in \eqref{eq:opoidring} and the definition of  
$\ptf_M$ in \eqref{eq:defptf}, $\ptf_M$ indeed defines a regular 
function on $\opoid$. Now by  \cite[Prop. 6.8]{JKS3}, 
$\ptf$ is a cluster character.
\end{proof}

\begin{remark} In the case of a connected positroid, Canakci--King--Pressland \cite{CKP}
show that when  $T=T_\maxNC$ and $M=M_I$, $\ptf_M$
is the same as  the twist defined by Muller-Speyer \cite{MS17} applied to 
the minor $\Delta_I$. This result was generalised in \cite{JKS3} to show 
that $\ptf_M$ is the same as the Muller-Speyer twist applied to $\Psi_M$
for an arbitrary module $M$, when $B=C$.
It would be interesting to explore the relation between  $\ptf_M$ and 
the twist defined by Muller-Speyer \cite{MS17} for an arbitrary positroid.
\end{remark}

\subsection{$\GP B$ and positroid varieties} \label{SEC:GPpositrod}
Denote by $\ideal_B$ the ideal in $(\algD)_B$ generated by $\ideal$, which 
can also be understood as the localisation of $\ideal$ at $\{ \Psi_{B'}\colon B'\in \add B\}$.

\begin{proposition}\label{prop:Idealloc} Let $M\in \CM C\backslash \CM B$. Then $\Psi_M\in \ideal_B$.
\end{proposition}
\begin{proof} By Proposition \ref{prop:mainproperty}, there exists $r>0$ such that 
\[
\Psi_M\Psi_{B^r}=\sum_{X\in\catD} a_X\Psi_X.
\]
Any $X$ in the sum with $a_X\not= 0$ is an extension of $M$ and $B^r$. As 
$M\not\in \CM B$, none of the X in the sum are in $\CM B$, as otherwise it would
imply that $M$ is the kernel or cokernel of a homomorphism between two 
$B$-modules and thus lead to a contradiction that $M\not\in \CM B$. Therefore, 
$X\not\in \GP B$ and so $\Psi_X\in \ideal$.  Hence $\Psi_M\in \ideal_B$.
\end{proof}

Recall the cluster algbra $\clalg$ constructed in Section \ref{sec:6.3}. In particular,  it is 
a subalgebra of $\CC[\Gr(k, n)]$ generated by the cluster variables $\Psi_M$ for 
reachable rigid objects $M\in \GP B$. Therefore, 
\begin{equation}\label{eq:alggpGr}
\clalg\subseteq \algGP.
\end{equation}

Let $\algGPptf$, $\algCMptf$ be the counterparts of $\algGP$ and $\algCM$ 
defined by $\ptf$, respectively.
Via the cluster character $\ptf$,  $M_\maxWS$ determines a seed 
$(\{\ptf_{M_I}\colon I\in \maxWS\}, B_{M_\maxWS}) $. The algebraic independence of 
the variables in the seed follows from \cite[Prop 8.3]{JKS3}, and we have 
$\ptf_{M_I}=\Delta_I^{-1}$ for any $M_I\in \add B$. 
Denote by $\clalgptf$ the cluster algebra defined by the seed. 
Similar to \eqref{eq:alggpGr}, 
\[
\clalgptf\subseteq \CC[\Gr(k, n)]_B.
\]

As the $k$-sets in $\nlace$ are the labels of the boundary faces of a Postnikov diagram 
(see \cite{OPS, Pos}), we say that the cluster variable $\Psi_{M_I}=\Delta_I$ is \emph{boundary} 
if $I\in \nlace$ (equivalently, if $M_I\in \add B$). 

The algebra $(\algGP)_B$ is the localisation of $\algGP$ at the boundary cluster 
variables $\{\Psi_{M_I}\colon M_I\in \add B\}$. Similarly, we can localise $\algGPptf$, $\algCMptf$ and 
$\clalgptf$ at the boundary cluster variables $\{\ptf_{M_I}=\Delta_I^{-1}\colon M_I\in \add B\}$, 
and denote the algebra obtained by $(\algGPptf)_{B^{-1}}$, $(\algCMptf)_{B^{-1}}$ and 
$(\clalgptf)_{B^{-1}}$, respectively. 

There is a natural projection, 
\begin{equation}\label{eq:pmap}
p\colon \CC[\Gr(k, n)]_B \to \CC[\opoid].
\end{equation}

\begin{theorem} \label{thm:cat} We have the following.
\begin{enumerate}
\item \label{itm:newcat1} $\algGP=\clalg$ and $\clalg\cong \clalgptf$ as cluster algebras.

\item The  map $p$ restricts to an isomorphism from $(\algGP)_B$ to $\CC[\opoid]$ and 
an injection  from $\algGP$ to $ \CC[\Pi] $.

\item The  map $p$ restricts to an isomorphism from $(\algGPptf)_{B^{-1}}$ to $\CC[\opoid]$ and 
an injection  from $\algGPptf$ to $ \CC[\opoid] $.

\end{enumerate}
\end{theorem}

\begin{proof} 
Let $\maxNC$ be a maximal collection of non-crossing $k$-sets in the positroid 
$\positr$.

(1) 
By \cite[Thm 9.11]{JKS3} and Corollary \ref{rem:char}, $\Psi_M$ for any $M\in \GP B$ 
is a Laurent polynomial in the cluster associated to any reachable cluster tilting object. Therefore, $\algGP$ is a subalgebra of the upper cluster algebra of $\clalg$. 

Following Theorem \ref{thm:locacy},  
$\clalg$ is locally acyclic, and so, by \cite[Thm 4.1]{Mu}, $\clalg$ is the same as its upper 
cluster algebra. Therefore $\algGP\leq \clalg$, which together  with \eqref{eq:alggpGr} implies
 $\algGP=\clalg$.

Note that $\ptf$ is a cluster character, for $U=M_\maxNC$, 
the cluster $\{\ptf^U_{M_I}\colon I\in \maxNC\}$ satisfies the mutation relations 
determined by the mutations of $U$, and so does $\{\Psi_{M_I}\colon I\in \maxNC\}$.
Therefore $\clalg\cong \clalgptf$ as cluster algebras (see \cite[Prop. 8.1]{JKS3}).

(2)
Note that $\CC[\Pi]=\CC[\Gr(k, n)]/\langle \{\Delta_I\colon I\not\in \CM B\} \rangle$, and 
 $(\algD)_{B}=\CC[\Gr(k, n)]_B$ (see \eqref{eq:algDCGratB}). By
 Proposition \ref{prop:Idealloc}, 
the ideal of $\CC[\Gr(k, n)]_B$ generated by  \[\{\Delta_I\colon I\not\in \CM B\}\] is contained in $\ideal_B$, 
hence there is a surjective map 
\[\xymatrix{
\sigma\colon \CC[\opoid]  \ar[r] &(\algD)_B/\ideal_B.
}\]
Now by (1) and Proposition \ref{prop:ideal}, 
the restriction of the composition $\sigma p$ to $(\algGP)_B$ is an isomorphism. 
Also, following (1), $(\algD)_B/\ideal_B$ 
is an integral domain 
and, by \eqref{eq:algDCGratB}, it is finitely generated. Therefore, 
$\spec (\algD)_B/\ideal_B$ and thus $\spec \algGP$ are irreducible. Furthermore, they 
have  dimension $ |\maxWS|$, following (1) that $\algGP$ is a cluster algebra of rank $|\maxNC|$.
Note that $\Pi$ and thus $\openpv$  are irreducible (\cite[Thm 5.9]{KLS}), 
and have dimension $|\maxNC|$ (see \eqref{eqPositroidDim}). Therefore, 
$\sigma$ is an isomorphism and so
\[
(\algGP)_B\cong\CC[\opoid].
\]
Hence, (2) follows.

(3) Recall that $\ptf_M=\Psi_{\Omega M}\Psi^{-1}_{P_M}$ and that $\Omega$ 
induces an equivalence on the stable category $\GP B/\add B$. Therefore 
$(\algGPptf)_{B^{-1}}=(\algGP)_B$. Thus, 
following (2),   \[(\algGPptf)_{B^{-1}}\cong \CC[\opoid],\]
and we have the injective map as described in (3). 
\end{proof}

\begin{corollary}
There is an injection $\CC[\opoid]\rightarrow \CC[\Gr(k,n)]_B$ so 
that the composition $\CC[\opoid]\rightarrow \CC[\Gr(k,n)]_B\rightarrow \CC[\opoid]$ 
is the identity on $\CC[\opoid]$.
\end{corollary}

For $a=(a_I)\in \ZZ^{\nlace}$, let \[\Delta_B^a=\prod_{I\in \nlace} \Delta_I^{a_I}.\]
\begin{corollary} The algebra $\CC[\opoid]$ has the following basis, 
\[\{\Psi_M\Delta_B^a\colon \Psi_M\in \basGP, \text{$M$ has no summand from $\add B$ } \text{and }
a=(a_I)\in \ZZ^n\}.
\]
\end{corollary}
\begin{proof}
It follows from Theorem \ref{thm:cat} (2) and Theorem \ref{thm:genericbasis}.
\end{proof}

\begin{remark}
Theorem \ref{thm:cat} provides an alternative proof of Galashin--Lam's Theorem \cite{GL}, 
which confirms Muller-Speyer's Conjecture 
 \ref{conj:MS}. 
\end{remark}
\begin{remark}
Note that $\clalgptf\subseteq \algGPptf$. So Theorem \ref{thm:cat}
reveals two distinct cluster subalgebras in $\CC[\opoid]$. 
We remark that the inclusion becomes an equality, when $\Sigma U$ is reachable from $U$
for any reachable cluster tilting object in $\GP B$.
\end{remark}

\section{Comparing cluster structures and sets of generators} \label{sec:ex}

In this section, we discuss distinct cluster structures on the coordinate rings $\CC[\Pi]$ and $\CC[\opoid]$ and 
their generators. 

\subsection{Cluster structures}
Theorem \ref{thm:cat} shows that $\GP B$
provides a categorification of two distinct cluster structures on $\CC[\opoid]$ via the cluster 
characters $\Psi$ and $\ptf$. 
The second author constructed a subcategory $R(v, w)$ \cite{Rio} associated to certain 
permutations $v, w$, which together with \cite[Thm. A]{SSB} provides a distinct categorification of a cluster structure on $\CC[\opoid]$. In the example below, we explicitly construct the cluster algebras, arising 
from $\GP B$ and $R(v, w)$ associated to a positroid variety $\Pi$. These cluster structures 
are not isomorphic; however, they become isomorphic after localisation. 

For any $i\in [n-1]$, denote by $\sigma_i$ the transposition on $[n]$ swapping $i$ and $i+1$.  

\begin{example} \label{ex:lastex}
Let $\sigma=(135)(264)$, which is a decorated permutation of type $(k, n)$ with 
$k=3$ and $n=6$. The associated necklace is 
\[\nlace = \{124, 234, 346, 456, 256, 126\},\] 
and the positroid variety $\poid$  is defined 
by $\Delta_{123} = \Delta_{345} = \Delta_{156} = 0.$  
We have 
\[B = \oplus_{I\in \nlace} M_I,\] and  
the  indecomposable objects in $\GP B$ 
\[\bl{M_{124}}, \bl{M_{234}}, \bl{M_{346}}, \bl{M_{456}}, \bl{M_{256}}, \bl{M_{126}}, \bl{M_{246}}, 
\bl{X},\]
where $X$ is the indecomposable rank two module with profile $=\frac{135}{ 246}$.
The modules $M_{246}$ and $X$ form a mutation pair and are the only two mutable modules
in $\GP B$.
 
We refer the reader to the AR-quiver \cite[Fig. 10]{JKS1}
for the two mutation sequences between $X$ and $M_{246}$ and the short 
exact sequences between rank one modules $M_I$ and $M_J$ discussed below for 
crossing sets $I$ and $J$.

Lifting Leclerc's category  $\catC_{v, w}$ \cite{Lec} to $\CM C$, the second 
author constructed a subcategory $\rich$ \cite{Rio}. In this example, 
$v=w_0^K\sigma_3$ and $w=w_0\sigma_2\sigma_4$, where $w_0$
is the longest element in the symmetric group of permutations of the set 
$[6]$, and $w_0^K$
is the longest element in the Young subgroup associated to $K=[6]\backslash \{k\}$. 
The indecomposable objects in $\rich$ are
\[\bl{M_{124}}, \bl{M_{234}}, \bl{M_{245}}, \bl{M_{456}}, \bl{M_{146}}, \bl{M_{126}}, \bl{M_{246}}, \bl{M_{145}}\]
where $\bl{M_{246}}$ and $\bl{M_{145}}$ are the mutable modules.
By inspection, $\rich$ is cluster category (see \cite{Rio} for general results), and thus give rise to a cluster algebra 
$\mathsf{A_{Rvw}}$. Denote by $(\mathsf{A_{Rvw}})_P$ the localisation of $\mathsf{A_{Rvw}}$ at minors corresponding to the projective objects (i.e. non-mutable modules) $M_I\in \rich$.

The algebra $\clalg$ is generated by 
$\{\Delta_I\colon M_I\in \GP B\}\cup \{\Psi_X\}$ with mutation relation
\begin{equation}\label{eq:gen1}
\Delta_{246}\Psi_X=\Delta_{124}\Delta_{256}\Delta_{346}+\Delta_{126}\Delta_{234}\Delta_{456}; 
\end{equation}
while $\mathsf{A_{Rvw}}$ is generated by 
$\{\Delta_I\colon M_I\in \rich\}$ with mutation relation
\begin{equation}\label{eq:gen2}
\Delta_{246}\Delta_{145}=\Delta_{124}\Delta_{456}+\Delta_{146}\Delta_{245}.
\end{equation}
Therefore they are both of type $\mathbb{A}_1$, but they are not isomorphic
as cluster algebras. However, they become isomorphic 
as algebras after localisation. We give some more details on the isomorphism. 

We have  the Pl\"{u}cker  relation,  
\[
\Delta_{245}\Delta_{346} = \Delta_{234}\Delta_{456}+ \Delta_{246}\Delta_{345}.
\]
and when restricted to $\poid$, 
\begin{equation}\label{eq:pl0}
\Delta_{245}\Delta_{346} = \Delta_{234}\Delta_{456}.
\end{equation}
Similarly, when restricted to $\poid$, 
\begin{equation}\label{eq:pl}\Delta_{146}\Delta_{256} = \Delta_{126}\Delta_{456}~ 
\text{ and } ~\Delta_{146}\Delta_{125}=\Delta_{126}\Delta_{145}.
\end{equation}
Note that $\dim \Ext^1(M_{346}, M_{125})=1$ and we have short exact sequences, 
\begin{equation*}
\xymatrix{
0\ar[r] &M_{346} \ar[r] & X \ar[r] & M_{125}\ar[r] &0,}
\end{equation*}
\begin{equation*}\xymatrix{
0\ar[r] & M_{125} \ar[r] & M_{126}\oplus M_{345} \ar[r] & M_{346}\ar[r]&0. 
}
\end{equation*}
Therefore, by \eqref{eq:pl} and the fact that $\Psi$ is a cluster character, when restricted to $\poid$,	
\begin{equation}
\label{eq:pl2}\Delta_{146}\Psi_{X} = \Delta_{126}\Delta_{346}\Delta_{145}.
\end{equation}
Following the relations \eqref{eq:pl0}-\eqref{eq:pl2}, we have 
\[(\clalg)_B= (\mathsf{A_{Rvw}})_P=\CC[\opoid].\]
\end{example}

\begin{remark}\label{rem:GIB}
The permutation $\sigma$ in Example \ref{ex:lastex} also determines an \emph{reverse} necklace (see e.g. \cite{FSB}) 
\[\nlace^{\mathrm{op}} = \{456, 146, 126, 236, 234, 245\},\]
which contains the labels of the boundary faces of a Postnikov diagram of type $\sigma$, following 
the left-source labelling convention of faces. It can also be described 
by replacing $I_{j+1}$ by $(I_j\backslash\{\sigma^{-1}(j)\})\cup \{j\}$, with 
$I_1=456$. 
The rank one modules $M_I$ with $I\in \nlace^{\mathrm{op}} $ are 
the indecomposable projective-injective objects the category  (see \cite{Pr2})
\[\gi B =\{M\in \CM B\colon \Ext^1(B\check{\;} , M)=0\},\]
where $B \,\check{}=\Hom_\ring(B, \ring)$, the dual of $B$ as a right $B$-module. 
In this case, 
 the  indecomposable 
objects are  
\[
\bl{M_{456}}, \bl{M_{146}}, \bl{M_{126}}, \bl{M_{236}}, \bl{M_{234}}, \bl{M_{245}}, \bl{M_{246}}, \bl{Y},
\]
where  $Y$ is the indecomposable rank two module with profile $\frac{246}{135}$, and 
$\bl{M_{246}}$ and $\bl{Y}$ are the mutable modules. The category $\gi B$ also determines a cluster 
algebra $\mathsf{A_{GIB}}$ (see \cite{Pr} for more details), which is different from but isomorphic to $\clalg$.
Note that the images of  $\clalg$ and $\mathsf{A_{GIB}}$ via the two injective maps 
in Theorem \ref{thm:cat}  are, in general, distinct proper subalgebras of $\CC[\opoid]$.
\end{remark}

\begin{remark}
Continue from Example \ref{ex:lastex}. 
Let $a=w^{-1}w_0$ and $b=v^{-1}w_0$. 
This pair of permutations $(a, b)$ determines the same positroid as Example \ref{ex:lastex} and is exactly the pair considered in \cite[Ex. 1.1]{GL} for this positroid. 
In this case, one can also compute explicitly Leclerc's category $\catC_{a, b}$ and lift it to 
a subcategory $R(a, b)$ of $\CM C$, via the functor $\omega$ defined in \cite[\S 5]{JKS2}.
The category $R(a, b)$ is distinct from $R(v, w)$, $\GP B$ and $\gi B$. 
Indeed, the indecomposable objects in $\catC_{a, b}$  are
\[
M_{345}, M_{235}, M_{236}, M_{136}, M_{156}, M_{456}, M_{356}, Y,
\]
where $Y$ is the rank two module as in Remark \ref{rem:GIB}.
However, 
the three associated cluster algebras are isomorphic, 
because their mutation quivers are the same. In fact, by applying the twist 
defined by $a=\sigma_2\sigma_4$ to $R(a, b)$, we obtain $\gi B$.
\end{remark}

\subsection{Generators of $\CC[\opoid]$}
By \eqref{eq:opoidring}, $\{\Delta_I\colon M_I\in \CM B\}$ is a generating set of  the coordinate ring $\CC[\Pi]$
of the positroid variety $\Pi$ associated to $B$. After being localised at the boundary minors, 
some minors in the generating set can be expressed by the others. So $\CC[\opoid]$ can 
be generated by a smaller set. In this subsection, we explore efficient generating sets of  $\CC[\opoid]$. 

\begin{proposition} \label{rem:gr2n}
Let $k=2$ and $\Pi$ be positroid variety of type $(k, n)$. Then $\CC[\opoid]$ is generated by 
$
\{\Delta_I, ~ \Delta_J^{-1}\colon M_I\in \GP B, ~ M_J\in \add B\}.
$
\end{proposition}

\begin{proof}
Note that in the case of $k=2$,  
the indecomposable modules in $\CM C$ all have rank one 
and $\dim \Ext^1(M_I, M_J)\leq 1$ \cite{JKS1}.

Suppose that $M_I\in \CM B\backslash \GP B$. It suffices to show that $M_I$ is generated by $\{\Delta_I, ~ \Delta_J^{-1}\colon M_I\in \GP B, ~ M_J\in \add B\}$.
By the description of $\GP B$ in \eqref{eq:defpgb},  there exists 
$M_J\in \add B$ such that $\Ext^1(M_I, M_J)\not=0$. So 
$\dim\Ext^1(M_I, M_J)=1$. We have the following two nonsplit short exact sequences, 
\[
\xymatrix{
0\ar[r]& M_J\ar[r] & X\ar[r] & M_I\ar[r] &0,
}\]
\[
\xymatrix{
0\ar[r]& M_I\ar[r] & Y\ar[r] & M_J\ar[r] &0,
}\]
where $X= M_{L_1}\oplus M_{L_2} $ and $ Y=M_{L_3}\oplus M_{L_4}$ 
are direct sums of two rank one modules.
By \eqref{eq:extCB},  $X\in \CM B$. Furthermore, following Proposition \ref{prop:indstep},   
we have $X\in \GP B$. So, $M_{L_1}, M_{L_2}\in \GP B$. 
On the other hand, $Y\not\in \CM B$; otherwise it would lead to a contradiction, as 
the second sequence would have to split, due to the fact that $B$ is projective in $\CM B$. 
So at least one of $M_{L_3}, M_{L_4}$ is not  in $\CM B$. 

As $\Psi$ is a cluster character, we have
\[
\Psi_{M_I}\Psi_{M_J}=\Psi_X+\Psi_Y=\Psi_{M_{L_1}}\Psi_{M_{L_2}}+\Psi_{M_{L_3}}\Psi_{M_{L_4}},
\]
and when restricted to $\Pi$, 
\[
\Psi_{M_I}\Psi_{M_J}=\Psi_{M_{L_1}}\Psi_{M_{L_2}}.
\]
So $\Psi_{M_I}=\Psi_{M_J}^{-1}\Psi_{M_{L_1}}\Psi_{M_{L_2}}$, equivalently, 
$\Delta_{I}=\Delta_{J}^{-1}\Delta_{L_1}\Delta_{L_2}$, 
as required. 
\end{proof}

\begin{remark}\label{rem:final}
Proposition \ref{rem:gr2n} is not true in general. 
Consider $\GP B$ from Example \ref{ex:lastex} and use the same notation. We claim that $\Psi_{X}$ can not be 
expressed as a polynomial in the minors $\Delta_I$,  $\Delta^{-1}_J$ with $M_I\in\GP B$ 
and $M_J\in \add B$.  Indeed, first 
note that there are exactly two basic cluster tilting objects in $\GP B$, with 
the corresponding clusters 
\[\{\Delta_{I}\colon I\in \nlace\}  \cup\{\Delta_{246}\}
\text{ and }\{\Psi_{X}\}\cup \{\Delta_{I}\colon I\in \nlace\},\] 
which are both algebraically independent sets, and the former
consist of all minors $\Delta_I$ with $M_I\in \GP B$. 
The only essential relation between the cluster variables is the mutation relation, 
\[\Psi_{X}=\Delta_{246}^{-1}(\Delta_{126}\Delta_{234}\Delta_{456}+
\Delta_{346}\Delta_{256}\Delta_{124}),\]
which involves inverting the mutable variable $\Delta_{246}$. So the 
claim holds. 
\end{remark}


\begin{thebibliography}{88}

\newcommand{\au}[1]{\textrm{#1},}
\newcommand{\ti}[1]{\textit{#1},}
\newcommand{\jo}[1]{\textrm{#1}}
\newcommand{\vo}[1]{\textbf{#1}}
\newcommand{\yr}[1]{(#1)}
\newcommand{\pg}[1]{#1}
\newcommand{\pp}[2]{#1--#2}

\bibitem{ARS} Auslander M., Reiten I. and Smal{\o} S., 
Representation theory of Artin algebras,
Cambridge University Press, 1997.

\bibitem{AS} Auslander M. and Smal{\o} S. 
{\it Almost split sequences in subcategories}, J. Alg. 69 (1981), 426--454.

\bibitem{BKM} Baur K., King A. and Marsh B., {\it Dimer models and cluster categories of Grassmannians},
P.L.M.S. 113 (2016), 213--260.

\bibitem{BIRS} Buan A. B., Iyama O.,  Reiten I.,  Scott J., 
{\it Cluster structures for 2-Calabi--Yau categories and unipotent groups}, 
Comp. Math. {145} (2009), 1035--1079.

\bibitem{CKP} Canakci I., King A. and Pressland M., 
{\it Perfect matching modules, dimer partition functions and cluster characters}, Adv. Math. 443 (2024), 109570.

\bibitem{CB} Crawley-Boevey W.,  {\it On the exceptional fibres of Kleinian singularities}, 
Amer. J. Math. 122 (2000), 1027--1037.

\bibitem{CS} Crawley-Boevey W. and Schr\"{o}er J., 
{\it Irreducible components of varieties of modules}, 
J. reine angew. Math. 553 (2002), 201--220.

\bibitem{FZ1}  Fomin S. and Zelevinsky A., 
{\it Cluster algebras I: Foundations},
J.A.M.S. 15 (2002), 497--529.  

\bibitem{FSB} 
Fraser C. and Sherman-Bennett M., {\it  Positroid cluster structures from relabeled plabic graphs},
Alg. Comb. 5 (2022), 469--513.

\bibitem{FuKeller} Fu C. and Keller B., 
{\it On cluster algebras with coefficients and 2-Calabi--Yau categories},
T.A.M.S. 362 (2010), {859}--{895}.

\bibitem{GL} Galashin P. and Lam T., 
{\it Positroid varieties and cluster algebras},
Ann. Scient. \'{E}cole Norm. Sup. 56 (2023), 859--844.

\bibitem{GSXX} Geiss C. Labardini-Fragoso D. and Schr\"{o}er J.,
{\it Semicontinous maps on module varieties}, J. reine angew. Math. 816 (2024), 1--17.

\bibitem{GLS05} Geiss C., Leclerc B., and Schr\"{o}er J., {\it Semicanonical bases and preprojective algebras}, 
Ann. Sci. Ecole. Norm. Sup. 38 (2005), 193--253. 

\bibitem{GLS07} Geiss C., Leclerc B., and Schr\"{o}er J., {\it Semicanonical bases and preprojective algebras II: a multiplication formula}, Comp. Math. 143 (2007), 1313--1334.

\bibitem{GLS08} Geiss C., Leclerc B., and Schr\"{o}er J., {\it Partial flag varieties and preprojective algebras}, Ann. de l'Institut Fourier 58 (2008), 825--876.

\bibitem{GLS12}  Geiss C., Leclerc B., and Schr\"{o}er J., {\it Generic basis for cluster algebras and 
the chamber Ansatz}, J.A.M.S. 25 (2012),  21--76.

\bibitem{JKS1} Jensen B. T., King A. and Su X., 
{\it A categorification of Grassmannian cluster algebras},
P. L. M. S. 113 (2016), 185--212.

\bibitem{JKS2} {Jensen B.T., King A. and Su X.},
{\it Categorification and the quantum Grassmannian},
Adv. Math. 406 (2022), 108577.

\bibitem{JKS3} Jensen B. T., King A. and Su X., {\it Categorification and the mirror symmetry for Grassmannians}, arXiv: 2404.14572. 

\bibitem{KR1} Keller B. and Reiten I., {\it Cluster-tilted algebras are Gorenstein and stably Calabi-Yau}, Adv. Math. 211 (2007), 123--151.

\bibitem{KLS} Knutson A., Lam T. and Speyer D. E.,  {\it Positroid varieties: juggling and
geometry}, Comp. Math. 149 (2013), 1710--1752.

\bibitem{Lec} Leclerc B., {\it Cluster structures on strata of flag varieties},
Adv. Math. 300 (2016), 190--228.

\bibitem{LZ} Leclerc B. and Zelevinsky A., 
{\it Quasicommuting families of quantum Pl\"ucker coordinates}, 
Kirillov’s Seminar on Representation Theory, A.M.S. Transl. Ser. 2 181 (1998), 85--108.

\bibitem{Lus00} Lusztig G., 
\ti{Semicanonical bases arising from enveloping algebras} 
\jo{Adv. Math. } 151 (2000),
 \pp{129}{139}.

\bibitem{MR04} Marsh B. and Rietsch K., {\it Parametrizations of flag varieties}, 
Rep. Theory 8 (2004), 212--242.

\bibitem{Mu} Muller G, {\it Locally acyclic cluster algebras}, 
Adv. Math. 233 (2013), 207--247.
 
\bibitem{MS} Muller G. and Speyer E. E., {\it Cluster algebras of Grassmannians are locally acyclic}, 
P.A.M.S. 144 (2015), 3267--3281.

\bibitem{MS17} Muller G. and Speyer D.E., {\it The twist for positroid varieties}, 
P.L.M.S. 115 (2017), 1014--1071.
 
\bibitem{OPS}  Oh~S., Postnikov~A. and Speyer~D. E.,
{\it Weak separation and plabic graphs}, P.L.M.S. 110 (2015), 721--754.

\bibitem{Pal} Palu Y.,  
{\it Cluster characters for 2-Calabi--Yau triangulated categories}, 
Ann. Inst. Fourier 58 (2008), 2221-2248. 

\bibitem{PL} Plamondon, P. G., 
{\it Generic bases for cluster algebras from the cluster category,}
Int. Math. Res. Not. 10 (2013), 2368-2420.

\bibitem{Pos} Postnikov A., {\it Total positivity, Grassmannians and networks}, 2006. \\
http://math.mit.edu/~apost/papers/tpgrass.pdf.

\bibitem{Pr} Pressland M., {\it Calabi-Yau properties of Postnikov diagrams}, 
Forum Math., Sigma (10) 2022, e56.

\bibitem{Pr2} Pressland M., {\it Quasi-coincidence of cluster structures on positroid varieties}, \\ arXiv: 2307.13369v2.

\bibitem{Rio} Riordan L. {\it Cohen-Macaulay-modules and Schubert varieties}, PhD thesis at the University of Bath.

\bibitem{Scott} Scott J., 
{\it Grassmannians and cluster algebras},
P.L.M.S. (92) 2006, 345--380. 

\bibitem{SSB} Serhiyenko K. and Sherman-Bennett M., 
{\it Leclerc's conjecture on a cluster structure for type A Richardson varieties}, 
Adv. Math. 447 (2024), 109698.
\end{thebibliography}
\end{document}